\documentclass[pdflatex,sn-mathphys-num]{sn-jnl}

\usepackage{graphicx}%
\usepackage{multirow}%
\usepackage{amsmath,amssymb,amsfonts}%
\usepackage{amsthm}%
\usepackage{mathrsfs}%
\usepackage[title]{appendix}%
\usepackage{xcolor}%
\usepackage{textcomp}%
\usepackage{manyfoot}%
\usepackage{booktabs}%
\usepackage{algorithm}%
\usepackage{algorithmicx}%
\usepackage{algpseudocode}%
\usepackage{listings}%
\usepackage{subcaption}
\usepackage{textcomp}
\usepackage{arydshln}
\usepackage{tikz}
\usepackage{pgfplots}
\pgfplotsset{compat = newest}
\usetikzlibrary{calc}
\usetikzlibrary{patterns.meta}
\usetikzlibrary{decorations.pathmorphing,patterns}
\usetikzlibrary{math}
\usetikzlibrary{cd}

\theoremstyle{thmstyleone}%
\newtheorem{theorem}{Theorem}
\newtheorem{proposition}[theorem]{Proposition}%
\newtheorem{lemma}[theorem]{Lemma}
\newtheorem{corollary}[theorem]{Corollary}

\theoremstyle{thmstyletwo}%
\newtheorem{example}{Example}%
\newtheorem{remark}{Remark}%

\theoremstyle{thmstylethree}%
\newtheorem{definition}{Definition}%
\newtheorem{ass}{Assumption}

\begin{document}

\title[Hybrid OC]{Dynamics and Optimal Control of State-Triggered Affine Hybrid Systems}

\author*[1]{\fnm{William} \sur{Clark}}\email{clarkw3@ohio.edu}

\affil*[1]{\orgdiv{Department of Mathematics}, \orgname{Ohio University}, \orgaddress{\street{24 Race St}, \city{Athens}, \postcode{45701}, \state{Ohio}, \country{USA}}}

\abstract{
A study of the dynamics and control for linear and affine hybrid systems subjected to either temporally- or spatially-triggered resets is presented. Hybrid trajectories are capable of degeneracies not found in continuous-time systems namely beating, blocking, and Zeno. These pathologies are commonly avoided by enforcing a lower bound on the time between events. While this constraint is straightforward to implement for temporally-triggered resets, it is impossible to do so for spatially-triggered systems. In particular, linear spatially-triggered hybrid systems always posses trajectories that are beating and blocking while affine systems may also include Zeno trajectories.

The hybrid Pontryagin maximum principle is studied in the context of affine hybrid systems. The existence/uniqueness of the induced co-state jump conditions is studied which introduces the notion of strongly and weakly actuated resets. Finally, optimal control in the context of beating and Zeno is discussed. This work concludes with numerical examples.
}

\keywords{Hybrid systems, Optimal control, Linear quadratic regulator}

\maketitle

\section{Introduction}\label{sec:intro}
The linear quadratic regulator (LQR) is the canonical introduction to optimal control theory and was among the earliest problems in this field \cite{kalman1960contributions}. A reason why it is so ubiquitous is that it is both sufficiently complicated to model many practical problems while being simple enough that a reasonably complete theory has been developed, e.g. \cite{jurdjevic}. The problem in question has the form
\begin{equation}\label{eq:master_cost}
	\begin{split}
		\min_u \, &\frac{1}{2}\int_{t_0}^{t_f} \, \left( x^\top Q x + u^\top Ru + 2x^\top N u\right) \, dt \\ &\hspace{1.3in} + \frac{1}{2}x^\top(t_f)Fx(t_f),
	\end{split}
\end{equation}
such that $x(t)$ follows the linear control dynamics
\begin{equation}\label{eq:cont_linear}
	\dot{x} = Ax + Bu,
\end{equation}
where $x\in\mathbb{R}^n$ is the state and $u\in\mathbb{R}^m$ is the input control. 
The matrices $R$ and $F$ are positive-definite, and $Q$ is positive-semidefinite. A standard approach to solving this problem is by integrating a matrix Riccati equation backwards in time.

If the underlying control system were nonlinear, iterative LQR can be utilized by linearizing about a reference trajectory and updating accordingly \cite{Li2004IterativeLQ}. However, when the state is allowed to jump, linearization becomes difficult. A distinguished class of such systems are hybrid systems \cite{hybrid_csm}. For the purposes of this work, we will restrict attention to simple hybrid systems \cite{simple_hybrid}, which have the form
\begin{equation*}
	\begin{cases}
		\dot{x} = f(x,u), & (t,x)\not\in \Sigma, \\
		x^+ = \Delta(x^-), & (t,x)\in\Sigma,
	\end{cases}
\end{equation*}
where the super scripts denote the state immediately pre- and post-jump and $\Sigma\subset \mathbb{R}\times \mathbb{R}^n$ is some event surface. A linear hybrid system is one where the data above are all linear (this does not imply that the flow is linear in the initial conditions). These control systems have the form
\begin{equation}\label{eq:hybrid_cont_linear}
	\begin{cases}
		\dot{x} = Ax + Bu, & (t,x)\not\in \Sigma, \\
		x^+ = Cx, & (t,x)\in \Sigma,
	\end{cases}
\end{equation}
where $\Sigma\subset\mathbb{R}\times\mathbb{R}^n$. This begs the question: what happens to the linear quadratic regulator when the dynamics \eqref{eq:cont_linear} are replaced by \eqref{eq:hybrid_cont_linear}?

The Pontraygin maximum principle (which is a way to provide the Riccati equation for LQR) remains applicable for hybrid systems (see, e.g.~\cite{DMITRUK2008964, liberzon, 4303244, gp_hnp} to name only a few). Moreover, the hybrid maximum principle has been studied in the linear case as well - the hybrid linear quadratic regulator (hLQR) \cite{periodic_hLQR, hLQR_1999} and iterative hLQR for nonlinear systems \cite{ihLQR_council}, among others. 

An important caveat is that the hybrid maximum principle is only applicable when the trajectories are not \textit{Zeno} \cite{zeno_pmp}, i.e. the jumps are all uniformly separated in time (equivalently, the dwell time is bounded below). A standard assumption to avoid Zeno is for $\overline{\Delta(\Sigma)}\cap\Sigma = \emptyset$, where the overline denotes the closure \cite{goodman_poincare, GRIZZLE20141955}. This intersection being nonempty is problematic as it allows for states to be ``stuck'' on the event surface.
This is easy enough to enforce for temporally triggered systems by having $\mathcal{T} = \{t_i\}\subset\mathbb{R}$ be a uniformly separated discrete subset. Unfortunately, this is not (and, indeed, cannot be) guaranteed for the case of spatially-triggered systems. Clearly,
\begin{equation}\label{eq:intersection_beating}
	0 \in \underbrace{\left\{ x\in\mathbb{R}^n : \lambda^\top x = 0 \right\}}_{\Sigma} \cap \underbrace{\left\{ Cx\in\mathbb{R}^n : \lambda^\top x = 0 \right\}}_{\Delta(\Sigma)} \ne \emptyset,
\end{equation}
as both of these sets are linear subspaces. Moreover, as both of these spaces are $(n-1)$-dimensional subspaces (assuming that the matrix $C$ is non-degenerate), their intersection is (at least) $(n-2)$-dimensional. By its very nature, the system always has beating/blocking. A trajectory is beating (resp. blocking) if it encounters the event surface a finite (resp. infinite) number of times in zero time.
In light of this, the hybrid maximum principle is not immediately applicable to spatially triggered systems. 

Fortunately, Zeno can be ruled out for a reasonable class of spatially triggered hybrid systems as shown in Theorem \ref{thm:no_zeno} in Section \ref{sssec:zeno} below. This essentially follows from the fact that the origin is a fixed point under both the continuous and jump dynamics -- such a point is not considered to be Zeno in \cite{stability_zeno}. For linear hybrid systems, the fixed point of the jump map and continuous dynamics are (assuming sufficient non-degeneracy)
\begin{equation*}
	\mathrm{Fix}_\Delta := \left\{ (t,x)\in\Sigma : \Delta(x) = x\right\} 
		\subset \mathrm{Fix}_f := \left\{ (t,x) : f(x) = 0\right\}.
\end{equation*}
Having the fixed points of the continuous and discrete dynamics coincide makes the class linear hybrid systems too restrictive to reasonably approximate many hybrid systems. As such, we will also consider the hybrid \textit{affine quadratic regulator} (AQR). The spatially triggered hybrid affine dynamics are given by
\begin{equation*}
	\begin{cases}
		\dot{x} = Ax + Bu + b, & (t,x)\not\in\Sigma, \\
		x^+ = Cx^-, & (t,x)\in\Sigma,
	\end{cases}
\end{equation*}
Unlike the linear case, affine systems can easily exhibit Zeno, e.g. the bouncing ball \cite{life_zeno}. 

Let us assume, for the moment, that Zeno does not appear in the solution to the spatially triggered hLQR/hAQR problem. The hybrid maximum principle states that the co-states must jump such that both the symplectic structure and the control Hamiltonian are preserved. If $p\in\mathbb{R}^n$ denotes the co-state, then this manifests as
\begin{equation}\label{eq:co-state_jump}
	\begin{split}
		C^\top p^+ &= p^- + \varepsilon\cdot \lambda, \\
		H^+ &= H^-,
	\end{split}
\end{equation}
where $H$ is the system's Hamiltonian (as constructed by the maximum principle) and $\varepsilon$ is a multiplier to enforce its conservation. As the Hamiltonian is typically quadratic in the co-states, we typically expect either two or zero solutions for the multiplier in \eqref{eq:co-state_jump}.

While \eqref{eq:co-state_jump} is well-known, a careful study of its applicability and analysis of its solutions is lacking. For reasons discussed above (beating/blocking/Zeno and existence/uniqueness), it is not even clear whether or not this relation makes sense for linear/affine hybrid systems.
The contributions of this work are the following:
\begin{itemize}
	\item Study the dynamical properties of linear/affine hybrid systems and gain an understanding of their ``exceptional'' solutions. In particular, it shown that beating is unavoidable, blocking can be minimized, and Zeno cannot occur in linear hybrid systems. Although the existence of beating/blocking remains unchanged for affine hybrid systens, Zeno may now appear.
	\item Show that temporally- and spatially-triggered control systems are qualitatively distinct. The affine hybrid quadratic regulator for the temporally-triggered case can be cast as a ``periodic affine Riccati equation'' while this does \textit{not} hold for spatially-triggered systems.
	\item Study the existence/uniqueness of solutions to \eqref{eq:co-state_jump} for spatially triggered systems. A natural dichotomy of control systems appear when determining uniqueness to \eqref{eq:co-state_jump} - weakly and strongly actuated resets. 
	\item Extend \eqref{eq:co-state_jump} to the case when beating occurs.
	\item Demonstrate the importance of Zeno in optimal trajectories.
\end{itemize}

The layout of this paper is as follows: Preliminaries for hybrid dynamical systems along with their solution concept is reviewed in \S\ref{sec:prelim}. Section \ref{sec:LHDS} focuses on the dynamics of linear hybrid systems and the existence of exceptional solutions. Proposition \ref{prop:triv_blocking} provides verifiable conditions to rule out blocking, Theorem \ref{thm:no_zeno} shows that Zeno does not occur in linear hybrid systems. Section \ref{sec:affine} extends the analysis of the previous section to include affine hybrid systems. Theorems \ref{thm:typeI_Zeno} and \ref{thm:typeII_Zeno} present two qualitatively distinct cases of Zeno for affine hybrid systems. Section \ref{sec:hLQR} introduces hybrid control systems, the hybrid affine quadratic regulator, and the specialization of the hybrid Pontryagin maximum principle to these systems. 
Section \ref{sec:complications} examines the complications of both existence/uniqueness and that of exceptional solutions for spatially-triggered resets, and introduces the concepts of weakly and strongly actuated resets. Furthermore, it is demonstrated that for some affine systems, non-singular solutions do not exist.
Numerical examples for both temporally and spatially triggered systems are demonstrated in Section \ref{sec:examples}. Finally, conclusions are discussed in Section \ref{sec:conclusion}.
\section{Preliminaries on Hybrid Dynamical Systems}\label{sec:prelim}
As control systems can be viewed as a generalization of dynamical systems, it will be beneficial to review some aspects of hybrid dynamical systems as the definition of a hybrid trajectory is more nuanced than a continuous-time trajectory. Moreover, these trajectories are capable of exhibiting unique pathological properties.

Throughout this section, a hybrid dynamical system will have the form
\begin{equation}\label{eq:gen_hybrid}
	\begin{cases}
		\dot{x} = f(x), & x\not\in \Sigma \\
		x^+ = \Delta(x), & x\in \Sigma,
	\end{cases}
\end{equation}
where the superscript, $x^-$, will be repressed and
\begin{enumerate}
	\item $x\in M$ where $M$ is a manifold called the \textit{state-space},
	\item $\Sigma\hookrightarrow M$ is an embedded co-dimension 1 submanifold called the \textit{guard},
	\item $f:M\to TM$ is a smooth vector field, and
	\item $\Delta:\Sigma\to M$ is a smooth map such that $\Delta(\Sigma)\hookrightarrow M$ is a smooth embedded submanifold called the \textit{reset}.
\end{enumerate}
Although the smoothness assumption on the above data can be greatly relaxed, to e.g. Lipschitz, the systems considered here will all be smooth.
Two special cases of the above will be studied in this work: linear and affine hybrid systems. 
A hybrid dynamical system will be referred to as \textit{linear} if its data is all linear, i.e.
\begin{equation}\label{eq:linear_hybrid}
	\begin{cases}
		\dot{x} = Ax, & \lambda^\top x \ne 0 \\
		x^+ = Cx, & \lambda^\top x = 0,
	\end{cases}
\end{equation}
which means that the continuous dynamics are linear, the guard is a linear subspace, and the reset is a linear map. As a slight generalization, an affine hybrid system will be given by
\begin{equation}\label{eq:affine_hybrid}
	\begin{cases}
		\dot{x} = Ax + b, & x\in\Sigma \\
		x^+ = Cx, & x\not\in\Sigma,
	\end{cases}
\end{equation}
where $\Sigma$ is an affine guard, see Definition \ref{def:aguard} below. 
For these systems, $A$ and $C$ are $n\times n$ matrices, and $b$ and $\lambda$ are $n\times 1$ column vectors.
\begin{remark}
	Note that the space of solutions to a linear hybrid system need not form a linear space. As such, the term ``hybrid system with linear data'' is more accurate. 
\end{remark}
To define a solution to \eqref{eq:gen_hybrid}, the notation of a hybrid time domain is needed, see \cite{hybrid_gst}.

\begin{definition}[Hybrid time domain]
	A subset $E\subset \mathbb{R}\times \mathbb{N}$ is a compact hybrid time domain if
	\begin{equation*}
		E = \bigcup_{j=1}^{J-1} \, [t_j,t_{j+1}] \times \{j\},
	\end{equation*}
	for some finite sequence of times $t_0\leq t_1 \leq \ldots\leq t_J = t_f$. A set $E$ is a hybrid time domain if for all $(T,J)\in E$, $E\cap [0,T] \times \{0,\ldots,J\}$ is a compact hybrid time domain.
\end{definition}

\begin{definition}[Hybrid Arc]
	A function $\varphi:E\to M$ is a hybrid arc if $E$ is a hybrid time domain and for each $j\in\mathbb{N}$, the function $\varphi_j:[t_j,t_{j+1}]\to M$ is locally absolutely continuous.
\end{definition}
Two useful functions for probing hybrid time domains/arcs are the projections onto the time components:
\begin{equation*}
	\begin{tikzcd}
		& E \subset \mathbb{R} \times \mathbb{N} \arrow[dl, "\pi_1"] \arrow[dr, "\pi_2"]& \\
		\mathbb{R} & & \mathbb{N}
	\end{tikzcd}
\end{equation*}
We can now state what it means to be a solution to a hybrid system.
\begin{definition}[Solution]\label{def:solution}
	A hybrid arc $\varphi:E\to M$ is a solution to the hybrid system \eqref{eq:gen_hybrid} if 
	\begin{enumerate}
		\item for all $j\in\mathbb{N}$ such that $I^j = \pi_1\circ\pi_2^{-1}(j)$ has nonempty interior,
		\begin{equation*}
			\begin{gathered}
				\varphi(t,j) \not\in\Sigma, \quad t\in \mathrm{int}(I^j) \\
				\frac{d}{dt}\varphi(t,j) = f\left( \varphi(t,j)\right), \quad t\in I^j \ a.e.
			\end{gathered}
		\end{equation*}
		\item if $(t,j), (t,j+1)\in E$, then
		\begin{equation*}
			\varphi(t,j)\in\Sigma, \quad \varphi(t,j+1) = \Delta(\varphi(t,j)).
		\end{equation*}
	\end{enumerate}
\end{definition}
Hybrid arcs can be qualitatively pathological in ways that continuous arcs cannot be. Three such qualities are listed below and constitute ``exceptional hybrid arcs.''
\begin{definition}[Exceptional Arcs]\label{def:singular_arcs}
	Let $\varphi:E\to M$ be a solution arc to \eqref{eq:gen_hybrid}. The hybrid arc is called	
	\begin{enumerate}
		\item $k$-beating at time $t$ if $\pi_2\circ\pi_1^{-1}(t) = \{j,j+1,\ldots,j+k+1\}$,
		\item blocking at time $t$ if $\pi_2\circ\pi_1^{-1}(t) = \{j, j+1,\ldots\}$, and
		\item Zeno if $\pi_1(E)$ is bounded, $\pi_2(E)$ is unbounded, and $\pi_1\circ\pi_2^{-1}(j)$ has nonempty interior for infinitely many $j$'s. 
	\end{enumerate}
\end{definition}
The index being offset by one for $k$-beating is to disqualify single resets from qualifying as 1-beating. 
Requiring that the intervals $\pi_1\circ\pi_2^{-1}(j)$ have nonempty interior is to differentiate between blocking and Zeno - or chattering Zeno and genuinely Zeno in the language of \cite{chattering}.
\begin{definition}[Zeno Points]\label{def:zeno_point}
	A point $x_z$ is a \textit{Zeno point} for the hybrid system \eqref{eq:gen_hybrid} if there exists a Zeno hybrid arc $\varphi:E\to M$ and a sequence $t_j\in \pi_1\circ\pi_2^{-1}(j)$ such that
	\begin{equation*}
		\lim_{j\to\infty} \, \varphi(t_j,j) = x_z.
	\end{equation*}
	The collection of all Zeno points is called the Zeno set and is denoted by $\mathcal{Z}$.
\end{definition}

If a solution is neither beating nor blocking, replacing the hybrid time domain by its projection, $\pi_1(E)\subset\mathbb{R}$, introduces isolated points of ambiguity and can be resolved by having the arc be left/right-continuous.
\subsection{Continuation of Solutions}
Of the three types of exceptional arcs defined in Definition \ref{def:singular_arcs}, blocking and Zeno arcs have $\pi_1(E)$ bounded above by the blocking/Zeno time. Many natural and physical systems have solutions of these types, e.g. the bouncing ball - see Example \ref{ex:soZ} below. The origin is the Zeno point of this system and corresponds to the ball coming to rest on the table in finite time. Obviously once the ball comes to a rest, it remains still on the table. This presents a way to extend the solution past the Zeno time. A study of this procedure for mechanical systems is explored in \cite{life_zeno}.

For the purposes of this work, we will continue solutions past their termination point by utilizing extended hybrid time domains as introduced in \cite{extend_zeno}.
\begin{definition}[Extended hybrid time domain]\label{def:extended}
	An extended hybrid time domain is a subset $\tilde{E}\subset \mathbb{R}^+ \times \mathbb{N}\times\mathbb{N}$ of the form
	\begin{equation*}
		\tilde{E} = \bigcup_{j,k} \, [t_{j,k}, t_{j+1,k}] \times \{j\} \times\{k\},
	\end{equation*}
	such that
	\begin{enumerate}
		\item it is a concatenation of hybrid time domains, i.e. for $k\in\mathbb{N}$, the set 
		\begin{equation*}
			E_k := \left\{ (t,j) : (t,j,k)\in\tilde{E}\right\} \subset \mathbb{R}\times \mathbb{N}
		\end{equation*}
		is a (possibly empty) hybrid time domain, and
		\item the individual hybrid time domains are compatible, i.e.
		\begin{equation*}
			\lim_{j\to\infty} \, t_{j,k} = t_{1,k+1},
		\end{equation*}
		assuming the above limit makes sense.
	\end{enumerate}
\end{definition}
\begin{figure}
	\centering
	\begin{tikzpicture}
		\draw[->, thick] (0,0) -- (10,0);
		\draw[->, thick] (0,0) -- (0,6);
		\node[below] at (10,0) {$t\in\mathbb{R}$};
		\node[above left] at (0,6) {$j\in\mathbb{N}$};
		\foreach \i in {1,...,10}{
			\draw[thick] (-0.1,{5-4.5*pow(2,-\i+1)}) -- (0.1,{5-4.5*pow(2,-\i+1)});
			\draw[thick, blue] ({4.5-4*pow(2,-\i+1)},{5-4.5*pow(2,-\i+1)}) -- ({4.5-4*pow(2,-\i)},{5-4.5*pow(2,-\i+1)});
			\fill[radius=2pt, blue] ({4.5-4*pow(2,-\i)},{5-4.5*pow(2,-\i+1)}) circle;
		}
		\foreach \i in {1,2,3}{
			\node[left] at (0, {5-4.5*pow(2,-\i+1)}) {\i};
		}
		\node[left] at (0,5) {$\infty$};
		\draw[thick, red] (4.5,0.5) -- (6.5,0.5);
		\foreach \i in {1,...,10}{
			\fill[radius=2pt, red] (6.5, {5-4.5*pow(2,-\i+1)}) circle;
		}
		\draw[thick, teal] (6.5,0.5) -- (8,0.5);
		\fill[radius=2pt, teal] (8,0.5) circle;
		\draw[thick, teal] (8,2.75) -- (9,2.75);
		\draw[thick, dashed, teal] (9,2.75) -- (9.75,2.75);
		\node[blue] at (2,2) {$(t,j,1)$};
		\node[red] at (5.5,2) {$(t,j,2)$};
		\node[teal] at (8,2) {$(t,j,3)$};
	\end{tikzpicture}
	\caption{An extended hybrid time domain with three components. The first (blue) is Zeno and the second (red) is blocking.}
	\label{fig:extended_hybrid_time}
\end{figure}
A sample extended hybrid time domain is shown in Fig. \ref{fig:extended_hybrid_time}. The use of this method will be rather limited in this work and is primarily utilized for optimal control problems where the time horizon lies beyond the Zeno/blocking time, see \S\ref{subsec:Zeno_ex}.
\section{Dynamics of Linear Hybrid Systems}\label{sec:LHDS}
We begin our analysis on the (uncontrolled) dynamics of linear hybrid systems \eqref{eq:linear_hybrid}, in particular on the existence/exclusion of exceptional hybrid solutions to these systems. This is examined as a preliminary to hybrid control as a standard assumption in optimal control for hybrid systems is that the resulting hybrid arcs have a finite number of separated resets on finite time intervals \cite{gp_hnp}, e.g. for an arc $\varphi:E\to M$,
\begin{equation*}
	\pi_1(E) = [t_0,t_f], \quad \#\pi_2(E)<\infty, \quad \text{and} \quad \min_{j\in \pi_2(E)} \, \mathrm{length}(\pi_1\circ\pi_2^{-1}(j)) \geq \delta > 0.
\end{equation*}
As such, the three types of exceptional arcs introduce a level of inconvenience to the standard formulation off the hybrid Pontryagin maximum principle \cite{minimum_hybrid}. A standard technique to avoid these types of arcs is to have the reset move points away from the guard, i.e. $\overline{\Delta(\Sigma)}\cap\Sigma = \emptyset$ where the overline represents the closure. For spatially-triggered resets in linear hybrid systems \eqref{eq:linear_hybrid}, this intersection is generally a co-dimension 2 subspace \eqref{eq:intersection_beating}. Although this set is not empty, it does not immediately imply the presence of Zeno trajectories, only that some states are beating which is not fatal to the hybrid Pontryagin maximum principle. In fact, the main result of this section, Theorem \ref{thm:no_zeno}, states that Zeno is actually impossible for these types of systems.

For the remainder of this section, all hybrid systems will have the form of \eqref{eq:linear_hybrid}, i.e.
\begin{equation*}
	\begin{cases}
		\dot{x} = Ax, & \lambda^\top x\ne 0, \\
		x^+ = Cx, & \lambda^\top x = 0,
	\end{cases}
\end{equation*}
subject to the following assumption.
\begin{ass}
	The reset matrix, $C$, is invertible.
\end{ass}

\subsection{Beating and Blocking Sets}\label{subsec:bblocking}
The \textit{guard} for a linear hybrid system is denoted by
\begin{equation*}
	\Sigma = \{ x\in\mathbb{R}^n : \lambda^\top x = 0\}.
\end{equation*}
If $x\in\Sigma$, a reset occurs and the point gets mapped to $Cx$. If $Cx\in\Sigma$, then a second reset instantaneously occurs and the point moves to $C^2x$ which results in beating. This inspires the following definition.

\begin{definition}[Beating and blocking sets]
	Consider the nested subspaces:
	$$\Sigma_0 \supseteq \Sigma_1 \subseteq \Sigma_2 \subseteq \ldots,$$
	defined by
	\begin{equation*}
		\Sigma_k := \bigcap_{j=0}^k \, \left\{ x\in\mathbb{R}^n : \lambda^\top C^jx = 0\right\}.
	\end{equation*}
	The subspace $\Sigma_k$ is called the \textit{$k^{th}$-beating set}. The \textit{blocking set} is the subspace
	\begin{equation*}
		\Sigma_\infty := \bigcap_{k=0}^\infty \, \Sigma_k.
	\end{equation*}
\end{definition}
\begin{remark}
	The following is true about the beating/blocking sets.
	\begin{enumerate}
		\item If $x\in\Sigma_1$, then $x, Cx\in\Sigma$ and the point is (at least) 1-beating. Likewise, a point $x\in \Sigma_k$ is (at least) $k$-beating. 
		\item If $x\in\Sigma_\infty$, then $x\in\Sigma$ and for all $k$, $C^kx\in\Sigma$ and the trajectory is blocking.
		\item As the beating sets are all nested subspaces, they must terminate and there exists a finite $N$ such that $\Sigma_N = \Sigma_\infty$.
		\item All of these spaces are nonempty as $0\in\Sigma_\infty$.
	\end{enumerate}
\end{remark}
\begin{remark}
	If a trajectory is either blocking or Zeno to the origin, we will extend the solution via Definition \ref{def:extended} by setting $x(t>t^*,1,2) \equiv 0$ where $t^*$ is the blocking/Zeno time. This will be implicitly used in \S\ref{subsec:Zeno_ex} below.
\end{remark}
Allowing for beating, the ``true'' reset map for the dynamics \eqref{eq:linear_hybrid} is $\Delta:\Sigma\to\mathbb{R}^n$ given by
 \begin{equation*}
	\Delta(x) = \begin{cases}
		Cx, & x\in\Sigma_0\setminus\Sigma_1, \\
		C^2x, & x\in\Sigma_1\setminus\Sigma_2, \\ 
		C^3x, & x\in\Sigma_2\setminus\Sigma_3, \\
		\vdots
	\end{cases}
\end{equation*}
and will be referred to as the \textit{full reset map}. The purpose of this reformulation is to exclude exceptional arcs and to allow for solutions to be parameterized exclusively by their time component, $\pi_1(E)\subset\mathbb{R}$. 
Clearly, the reset map does not always have a clear extension to the blocking set, $\Sigma_\infty$, as $\lim_k C^kx$ need not be reasonable. This leads to the observation that \textit{for a general linear hybrid system, the full reset may not make sense on the entire guard.} However, the reset map will always have a clear extension to the origin (which always lies within the blocking set) by defining $\Delta(0) := 0$.
\begin{definition}[Trivially Blocking]
	The system \eqref{eq:linear_hybrid} is called \textit{trivially blocking} if $\Sigma_\infty = \{0\}$.
\end{definition}
When a system is trivially blocking, the full reset map can be extended to the entire guard via
 \begin{equation*}
	\Delta(x) = \begin{cases}
		Cx, & x\in\Sigma_0\setminus\Sigma_1, \\
		C^2x, & x\in\Sigma_1\setminus\Sigma_2, \\ 
		C^3x, & x\in\Sigma_2\setminus\Sigma_3, \\
		\vdots \\
		C^Nx, & x\in\Sigma_{N-1}\setminus \Sigma_k, \\
		0, & x\in\Sigma_k = \{0\},
	\end{cases}
\end{equation*}
where $N$ is the smallest integer such that $\Sigma_N = \Sigma_\infty$. Unfortunately, systems cannot be assumed to be trivially blocking as the beating sets are expected to have as large a dimension as possible.
\begin{proposition}\label{prop:size_beating}
	The dimension of the first-beating set is given by
	\begin{equation*}
		\dim \Sigma_1 = \begin{cases}
			\dim\Sigma = n-1, & \lambda^\top C = \alpha\lambda^\top, \\
			\dim\Sigma-1 = n-2, & \text{else},
		\end{cases}
	\end{equation*}
	where $\alpha\in \mathbb{R}$. 
	Additionally, if $\lambda^\top C = \alpha\lambda^\top$ (so $\lambda^\top$ is a left eigenvector of $C$), then $\Sigma_\infty = \Sigma_1$.
\end{proposition}
\begin{proof}
	This follows directly from the fact that
	\begin{equation*}
		\Sigma_1 = \left\{ x\in\mathbb{R}^n : \lambda^\top x = 0 = \lambda^\top Cx\right\}.
	\end{equation*}
	The dimension of this set is determined by whether or not the vectors $\lambda^\top$ and $\lambda^\top C$ are linearly independent.
\end{proof}
This argument can be repeated to determine when a system is trivially blocking. 
\begin{proposition}\label{prop:triv_blocking}
	The system \eqref{eq:linear_hybrid} is trivially blocking if and only if the hybrid rank condition is satisfied:
	\begin{equation}\label{eq:triv_blocking}
		\mathrm{rank} \, \begin{bmatrix}
			\lambda^\top \\
			\lambda^\top C \\
			\vdots \\
			\lambda^\top C^{n-1}
		\end{bmatrix} = n.
	\end{equation}
\end{proposition}
\begin{proof}
	This follows from a similar argument to Proposition \ref{prop:size_beating}. If there exists a $k$ such that 
	\begin{equation*}
		\lambda^\top C^k \in \mathrm{span}_\mathbb{R} \left\{ \lambda^\top, \lambda^\top C, \ldots, \lambda^\top C^{k-1}\right\},
	\end{equation*}
	then $\Sigma_k = \Sigma_{k-1} = \Sigma_\infty$.
\end{proof}
\begin{remark}
	The condition \eqref{eq:triv_blocking} is equivalent to the controllability matrix having full rank:
	\begin{equation*}
		\mathrm{rank} \, \mathcal{C}(C^\top,\lambda) = n, \quad \mathcal{C}(C^\top,\lambda) = \left[ \lambda, C^\top\lambda, \ldots, (C^\top)^{n-1}\lambda\right].
	\end{equation*}
	As such, \eqref{eq:linear_hybrid} is trivially blocking if and only if the pair $(C^\top,\lambda)$ is controllable.
\end{remark}
For an illustration of a trivially blocking system in $\mathbb{R}^3$, see Fig. \ref{subfig:triv_blocking}.
\begin{figure}
	\centering
	\begin{subfigure}[t]{0.45\textwidth}
		\centering
		\begin{tikzpicture}[scale=0.75]
			\draw [fill=blue, fill opacity = 0.4] (-3,-1) -- (3,-2) -- (5,1) -- (-1,2) -- cycle;
			\node [below right] at (3, -2) {$\textcolor{blue}{\Sigma}$};
			\draw [fill=red, fill opacity = 0.4] (0,-3.25) -- (2, -1) -- (2, 3.5) -- (0, 1.5) -- cycle;
			\node [right] at (2, 3.5) {$\textcolor{red}{C\Sigma}$};
			\draw[black, thick] (-2,0.5) -- (4,-0.5);
			\node [below right] at (4,-0.5) {$\Sigma_1$};
			\draw [black, thick] (0,-1.5) -- (2,1.5);
			\node [above right] at (2,1.5) {$\textcolor{black}{C\Sigma_1}$};
			\draw [black, thick] (0.667, 2.167) -- (1.333, -1.75);
			\node [above left] at (0.667, 2.167) {$C^2\Sigma_1$};
		\end{tikzpicture}
		\caption{An illustration of a trivially blocking linear hybrid system in three dimensions. For a vector $u\in\Sigma_1$, the set $\{u, Cu, C^2u\}$ forms a basis which is consistent with \eqref{eq:triv_blocking}.}
		\label{subfig:triv_blocking}
	\end{subfigure}
	\hfill
	\begin{subfigure}[t]{0.45\textwidth}
		\centering
		\begin{tikzpicture}
			\shade[ball color = gray!40, opacity = 0.25] (0,0) circle (2cm);
			\draw[thick, blue] (-2,0) arc (180:360:2 and 0.6);
			\draw[thick, blue, dashed] (-2,0) arc (180:0:2 and 0.6);
			\draw[thick, red, rotate=270] (-2,0) arc (180:360:2 and 0.8);
			\draw[thick, red, rotate=270, dashed] (-2,0) arc (180:0:2 and 0.8);
			\node[above right, red] at (0,2) {$\pi(C\Sigma)$};
			\node[blue] at (1.5,-1) {$\pi(\Sigma)$};
			\draw[fill] (-2,0) circle [radius=0.07];
			\draw[fill] (2,0) circle [radius=0.07];
			\node[right] at (2,0) {$\pi(\Sigma_1)$};
			\draw[fill, teal] (-0.77,-0.56) circle [radius=0.07];
			\draw[fill, teal] (0.77,0.56) circle [radius=0.07];
			\node[above right, teal] at (-0.77,-0.56) {$\pi(C\Sigma_1)$};
			\draw[fill, violet] (-0.575,1.4) circle [radius=0.07];
			\draw[fill, violet] (0.575,-1.4) circle [radius=0.07];
			\node[above left, violet] at (-0.6,1.4) {$\pi\left( C^2\Sigma_1\right)$};
		\end{tikzpicture}
		\caption{The trivially blocking system on $\mathbb{R}^3$ projected to $\mathbb{S}^2$ via Lemma \ref{lem:tau_projective}. Here, the map $\pi:\mathbb{R}^3\to\mathbb{S}^2$ normalizes the vector.}
		\label{subfig:projective_blowup}
	\end{subfigure}
	\caption{Illustration of a trivially blocking system and its projection as used in Theorem \ref{thm:no_zeno} below.}
	\label{fig:linear_hds}
\end{figure}
\subsection{Exclusion of Zeno Trajectories}\label{sssec:zeno}
Although linear hybrid systems necessarily have beating and blocking solutions, it turns out that (when the system is trivially blocking) Zeno never occurs. A necessary condition for Zeno from Definition \ref{def:singular_arcs} is for
\begin{equation*}
		\lim_{j\to\infty} \, \mathrm{length} \left\{\pi_1\circ\pi_2^{-1}(j)\right\} = 0.
\end{equation*}
To rule out Zeno, we will develop a lower bound between reset time. To do so, we define the invariant guards and the first-return time along with a key property of the first-return time. 
\begin{definition}[Invariant Guards]
	Call the subspace the invariant guard,
	\begin{equation*}
		\Sigma^A := \left\{ x\in\Sigma : e^{At}x\in\Sigma, \ \forall t\in\mathbb{R}\right\} = \left\{ x\in\Sigma : \lambda^\top Ax = 0\right\} \subset\Sigma,
	\end{equation*}
	along with the invariant beating/blocking sets
	\begin{equation}\label{eq:invariant_beating}
		\Sigma_k^A := \left\{ x\in \Sigma_k : e^{At}x\in\Sigma_k, \ \forall t\in\mathbb{R} \right\},
	\end{equation}
	for $k\in\mathbb{N}\cup\{\infty\}$.
\end{definition}
\begin{definition}[First-Return Time]\label{def:first_return}
	Define the function $\tau:\mathbb{R}^n \to \mathbb{R}^+\cup\{\infty\}$ to be the first-return time to $\Sigma$ from only the continuous dynamics,
	\begin{equation*}
		\tau(x) = \arg\min_{t\geq 0} \, \left\{\lambda^\top e^{At}x = 0\right\}.
	\end{equation*}
	If this condition is never satisfied, set $\tau(x) := \infty$.
\end{definition}
\begin{lemma}\label{lem:tau_projective}
	The map $\tau$ is projective, i.e. $\tau(sx)=\tau(x)$ for any $s\in\mathbb{R}\setminus\{0\}$. Moreover, this function is continuous away from the invariant guard, $\mathbb{R}^n\setminus \Sigma^A$.
\end{lemma}
\begin{proof}
	Projectivity follows from
	\begin{equation*}
		\lambda^\top e^{At}x = 0 \iff \lambda^\top e^{At}sx = 0,
	\end{equation*}
	for any $s\ne 0$. Continuity follows from applying the implicit function theorem to 
	\begin{equation*}
		g(t,x) = \lambda^\top e^{At}x, \quad \frac{\partial g}{\partial t} = \lambda^\top Ae^{At}x.
	\end{equation*}
	The partial derivative with respect to time is non-zero as long as $e^{At}x\not\in\Sigma^A$. As this set is invariant, the result follows.
\end{proof}
A Zeno trajectory of a trivially blocking system must limit to the origin. The key utility of the above lemma is that the origin may be blown up by passing to the real-projective space, see Fig. \ref{subfig:projective_blowup}. 
\begin{theorem}\label{thm:no_zeno}
	Suppose that the linear hybrid system \eqref{eq:linear_hybrid} is trivially blocking. Then for any initial condition $x_0\ne 0$, the resulting trajectory is not Zeno.
\end{theorem}
\begin{proof}
	Suppose the trajectory $\varphi:E\to\mathbb{R}^n$ is Zeno and let $\{s_k\}_{k=1}^\infty\subset [0,t]$ be the collection of reset times and $x_k = \varphi(s_k,k-1)\in\Sigma$ be the corresponding collection of reset locations. By being Zeno, the reset times satisfy
	\begin{equation*}
		\lim_{k\to\infty} \, s_{k+1}-s_k = \lim_{k\to\infty} \, \tau\left( \Delta(x_k)\right) = 0.
	\end{equation*}
	Let $\pi:\mathbb{R}^n\setminus\{0\}\to\mathbb{RP}^{n-1}$ be the canonical projection to the real projective space. Call the induced sequence $\hat{x}_k = \pi(x_k)$. By compactness of $\mathbb{RP}^{n-1}$, there exist a limit point, $\hat{x}_{k_j}\to \hat{x}$. The result follows if $\hat{\tau}\left(\hat{\Delta}(\hat{x})\right) >0$ where
	\begin{equation*}
		\hat{\tau}:\mathbb{RP}^{n-1}\to\mathbb{R}\cup\{\infty\}, \quad \hat{\Delta}:\pi\left(\Sigma\right) \to \mathbb{RP}^{n-1},
	\end{equation*}
	are the induced maps. Call $\hat{y}:= \hat{\Delta}(\hat{x})$.
	
	Suppose that there exists $z\in C\Sigma\setminus C\Sigma_1$ such that $\hat{y} = \pi(z)$. Then either $\Sigma_1 = \Sigma_1^A$ or $\Sigma_1 \ne \Sigma_1^A$. 
	\begin{enumerate}
		\item Suppose that $\Sigma_1 = \Sigma_1^A$. As this set is invariant, 
		\begin{equation*}
			e^{A\tau(z)}z \not \in \Sigma_1,
		\end{equation*}
		as $z\not\in\Sigma_1$. By linearity of the continuous flow and the projectivity of $\tau$,
		\begin{equation*}
			\tau(z) \geq \inf_{x\in C\Sigma\setminus C\Sigma_1} \, \tau(x) = 
			\inf_{\substack{x\in C\Sigma \\ x\perp \Sigma_1 \\ \lVert x\rVert = 1}} \, \tau(x).
		\end{equation*}
		This value is strictly positive by compactness and continuity of $\tau$.
		\item Suppose that $\Sigma_1 \ne \Sigma_1^A$. We can find a sequence $x_k\in C\Sigma\setminus C\Sigma_1$ such that $x_k\to \Sigma_1$ and $\tau(x_k)\to 0$. As this is insufficient to prohibit Zeno, we pass to the next subspace: Consider $z\in C\Sigma_1\setminus C\Sigma_2$. This leads to the dichotomy $\Sigma_2 = \Sigma_2^A$ or $\Sigma_2\ne \Sigma_2^A$. 
	\end{enumerate}
	The argument above is iterated. A positive return time is found if $\Sigma_k = \Sigma_k^A$ for some $k$. This is guaranteed to occur as the system is trivially blocking and $\Sigma_N = \Sigma_N^A = \{0\}$ for some finite $N$ large enough such that $\Sigma_N = \Sigma_\infty = \{0\}$.
\end{proof}
\section{Affine Hybrid Systems}\label{sec:affine}
In light of Theorem \ref{thm:no_zeno}, a hybrid system experiencing Zeno is not reasonably approximated by a linear hybrid system. A primary shortcoming of linear hybrid systems is that the (trivially) blocking set $\Sigma_\infty = \{0\}$ is also a fixed point of the continuous dynamics. As such, if a trajectory were to approach this set, the continuous dynamics would necessarily slow down enough to prevent reaching that set in finite time. This deficiency in linear hybrid systems motivates the study of affine hybrid systems. 

An affine hybrid system \eqref{eq:affine_hybrid} is a generalization of a linear hybrid system to allow for a bias in both the dynamics and the guard. A bias in the reset is not considered as may be removed via a coordinate change. For the remainder of this section, all hybrid systems will have the form
\begin{equation*}
	\begin{cases}
		\dot{x} = Ax + b, & (t,x)\not\in\Sigma, \\
		x^+ = Cx, & (t,x)\in\Sigma,
	\end{cases}
\end{equation*}
where $\Sigma\subset \mathbb{R}\times\mathbb{R}^n$ is a (time-dependent) affine guard.
\begin{definition}[Affine guard]\label{def:aguard}
	A (time-dependent) affine guard, $\Sigma\subset\mathbb{R}\times\mathbb{R}^n$ will have the form of one of the following.
	\begin{enumerate}
		\item An affine subspace, i.e. there exists $(\lambda_0,\lambda)\in\mathbb{R}\times\mathbb{R}^n$ and a number $a\in\mathbb{R}$ such that
		\begin{equation*}
			\Sigma = \left\{ (t,x)\in\mathbb{R}\times\mathbb{R}^n : \lambda_0t + \lambda^\top x = a \right\}.
		\end{equation*}
		\item Half of an affine subspace, i.e. there exists $(\lambda_0,\lambda), (\nu_0,\nu)\in\mathbb{R}\times\mathbb{R}^n$ and numbers $a,b\in\mathbb{R}$ such that
		\begin{equation*}
			\Sigma = \left\{ (t,x)\in\mathbb{R}\times\mathbb{R}^n : \begin{array}{c}
				\lambda_0t + \lambda^\top x = a \\
				\nu_0t + \nu^\top x < b
			\end{array} \right\}.
		\end{equation*}
		\item A (disjoint) union of the above, i.e. $\Sigma = \cup_\alpha \Sigma_\alpha$.
	\end{enumerate}
	If $\lambda_0=\nu_0$, the guard will be called time-independent.
\end{definition}
\begin{remark}
	Guards that are an affine subspace are a straightforward extension of linear guards discussed in \S\ref{sec:LHDS}. Guards that are half of an affine subspace are required to formulate mechanical impact systems, e.g. the pedagogical bouncing ball example shown in Example \ref{sec:typeII} below. Finally, periodically excited hybrid systems have a guard $\Sigma = \kappa\mathbb{Z}\times\mathbb{R}^n$ is a union of affine subspaces as studied in, e.g. \cite{periodic_hLQR}.
\end{remark}
Although affine guards come is different flavors, their tangent spaces can be identified in the same way:
\begin{equation*}
	(\delta t, \delta x) \in T_{(t,x)}\Sigma \iff \lambda_0\delta t + \lambda^\top \delta x = 0.
\end{equation*}
This will be used below in \S\ref{sec:hLQR} for optimal control.

The analysis in \S\ref{subsec:bblocking} carries over naturally to affine systems and discussing beating and blocking for these systems is omitted as it would be largely redundant. However, a principal difference between linear and affine systems is that the latter may have Zeno solutions as Lemma \ref{lem:tau_projective} is no longer true. It turns out that affine systems may posses two qualitatively distinct depending on whether the guard is an affine subspace (first-order) or half of an affine subspace (second-order). An illustration of these two cases is shown in Fig. \ref{fig:planar_zeno_types}. It is clear that if the guard has the form $\kappa\mathbb{Z}\times\mathbb{R}^n$, Zeno is impossible.

\begin{figure}
	\centering
	\begin{subfigure}[t]{0.45\textwidth}
		\centering
		\begin{tikzpicture}[scale=0.9]
			\draw [<->, thick, blue] (-1,0) -- (5,0);
			\draw [<->, thick, red] (0,-1) -- (0,5);
			\node [below right] at (5,0) {$\textcolor{blue}{\Sigma}$};
			\node [above left] at (0,5) {$\textcolor{red}{\Delta(\Sigma)}$};
			\draw [fill, black] (0,0) circle [radius=0.07];
			\foreach \x in {0.5, 1, 1.5, 2, 2.5, 3, 3.5, 4}{
				\foreach \y in {0.5, 1, 1.5, 2, 2.5, 3, 3.5, 4}{
					\pgfmathsetmacro{\vx}{1/5}
					\pgfmathsetmacro{\vy}{-1/5}
					\draw[->] (\x, \y) -- (\x + \vx, \y + \vy);
				}
			}
			\draw[violet, very thick, ->] (1,3) -- (4,0);
			\draw[violet, very thick, ->] (0,2) -- (2,0);
			\draw[violet, very thick, ->] (0,1) -- (1,0);
			\draw [violet, thick, dashed, ->] (4,0) to [out=-90, in=-45] (-0.5, -0.5) to [out=135, in=-180] (0,2);
			\draw [violet, thick, dashed, ->] (2,0) to [out=-90, in=-45] (-0.25, -0.25) to [out=135, in=-180] (0,1);
		\end{tikzpicture}
		\caption{An example of a planar affine hybrid system with Type I Zeno.}
		\label{fig:first_order_zeno}
	\end{subfigure}
	\hfill
	\begin{subfigure}[t]{0.45\textwidth}
		\centering
		\begin{tikzpicture}[scale=0.9]
			\draw [<->, thick, black] (-1,0) -- (5,0);
			\draw [->, thick, red] (0,0) -- (0,3);
			\draw [->, thick, blue] (0,0) -- (0,-3);
			\draw [fill, black] (0,0) circle [radius=0.07];
			
			\node [below left] at (0,-3) {$\textcolor{blue}{\Sigma}$};
			\node [above left] at (0,3) {$\textcolor{red}{\Delta(\Sigma)}$};
			\foreach \x in {0.5, 1, 1.5, 2, 2.5, 3, 3.5, 4}{
				\foreach \y in {-2.5, -2, -1.5, -1, -0.5, 0, 0.5, 1, 1.5, 2, 2.5}{
					\pgfmathsetmacro{\vx}{\y/5}
					\pgfmathsetmacro{\vy}{-1/5}
					\draw[->] (\x, \y) -- (\x + \vx, \y + \vy);
				}
			}
			\draw [violet, very thick, domain=0:4.828, samples=40, ->]
			plot ({-1/2*\x*\x+2*\x+2}, {-\x+2});
			\draw [violet, very thick, domain=0:2.828, samples=40, ->]
			plot ({-1/2*\x*\x+1.414*\x}, {-\x + 1.414});
			\draw [violet, very thick, domain = 0:1.414, samples=40, ->]
			plot({-1/2*\x*\x + 0.707*\x}, {-\x + 0.707});
			\draw [violet, thick, dashed, ->] (0,-2.828) to [out=180, in=-180] (0, 1.414) ;
			\draw [violet, thick, dashed, ->] (0,-1.414) to [out=180, in=-180] (0, 0.707) ;
		\end{tikzpicture}
		\caption{An example of a planar affine hybrid system with Type II Zeno.}
		\label{fig:second_order_zeno}
	\end{subfigure}
	\caption{Two qualitatively distinct manifestations of Zeno in planar affine hybrid systems.}
	\label{fig:planar_zeno_types}
\end{figure}

To assist with exposition, the remainder of this section will make the following assumption. 
\begin{ass}
	The affine hybrid systems consider below will be planar, i.e. $x\in\mathbb{R}^2$.
\end{ass}
\subsection{Type I Zeno}
Consider the affine hybrid system \eqref{eq:affine_hybrid} where $\Sigma = \{x:\lambda^\top x = 0\}$ and $\lambda^\top$ is not a left eigenvector of $C$.
Zeno trajectories are possible as shown in the following example.
\begin{example}[First-Order Zeno]\label{ex:foZ}
	The blocking set must lie within the intersection of the guard and its image. If the intersection is transverse, then a Zeno trajectory will be (approximately) piece-wise linear. Consider the following dynamics:
	\begin{equation*}
		\begin{cases}
			\left. \begin{array}{l}
				\dot{x} = a \\
				\dot{y} = -b
			\end{array}\right\} & x,y > 0 \\
			\left. \begin{array}{l}
				x^+ = 0 \\
				y^+ = cx^-
			\end{array}\right\} & y^- = 0
		\end{cases}
	\end{equation*}
	where $a,b>0$ and $0<c<1$, see Fig. \ref{fig:first_order_zeno} for its phase portrait.
	Zeno occurs when $ca<b$ with Zeno time
	\begin{equation*}
		\zeta_1(x_0, y_0) = \frac{y_0}{b} + \frac{c}{b-ca}\left(x_0 + \frac{a}{b}y_0\right).
	\end{equation*}
	The Zeno time is linear in the initial conditions.
\end{example}
A natural question to ask is what conditions on the data for an affine hybrid system, $(A,b,C,\lambda)$, results in Zeno as a generalization of the above example?
\begin{theorem}\label{thm:typeI_Zeno}
	Let $(A,b,C,\lambda)$ be the data for an affine hybrid system such that $\lambda^\top$ is not a left eigenvector of $C$ and $\lambda^\top b \ne 0$. Let $v\in\Sigma$ be a unit vector spanning $\Sigma$. Then, for an initial condition $x(0)\ne 0$, the resulting trajectory is Zeno if
	\begin{equation}\label{eq:type1_zeno}
		-1 < v^\top \left[ Cv - \frac{\lambda^\top Cv}{\lambda^\top b}b\right] < 1.
	\end{equation}
\end{theorem}
\begin{proof}
	For an initial condition, $x_0\in\mathbb{R}^2$, let $\{x_k\}_{k=1}^\infty$ be the collection of points in the trajectory that lie on the guard, i.e. for the trajectory $x(\cdot):E\to \mathbb{R}^n$, let
	\begin{equation*}
		x_k := x(t_k, k), \quad t_k = \max_t \, \left\{ t\in \pi_2^{-1}(k)\right\}.
	\end{equation*}
	Then, the trajectory is Zeno if 
	\begin{equation}\label{eq:zeno_conditions}
		\lim_{k\to\infty} \, x_k = 0, \quad \text{and} \quad \sum_{k=1}^\infty \, \tau(Cx_k) < \infty,
	\end{equation}
	where $\tau:\mathbb{R}^2\to\mathbb{R}$ is the first-return time to $\Sigma$ under the dynamics $\dot{x} = Ax+b$. By modifying the proof of Lemma \ref{lem:tau_projective}, this map is continuous near the origin when $\lambda^\top b \ne 0$. Additionally, its derivative can be calculated by the implicit function theorem:
	\begin{equation*}
		\left.\frac{d}{ds}\right|_{s=0} \, \tau(Csv) = -\frac{\lambda^\top v}{\lambda^\top b}.
	\end{equation*}
	Let $R:\mathbb{R}\to\mathbb{R}$ be the first-return map defined along the guard, i.e.
	\begin{equation*}
		R(s) = v^\top\varphi(\tau(Csv), Csv),
	\end{equation*}
	where $\Sigma = v\mathbb{R}\subset\mathbb{R}^2$ and $\varphi:\mathbb{R}\times\mathbb{R}^n\to\mathbb{R}^n$ is the flow of $\dot{x} = Ax+b$. Then the first part of \eqref{eq:zeno_conditions} is satisfied if $|R'(0)|<1$. Calculating this, we find
	\begin{equation*}
		\begin{split}
			R'(0) &= \left.\frac{d}{ds}\right|_{s=0} \, v^\top \varphi(\tau(Csv), Csv) \\
			&= v^\top \left( \frac{\partial \varphi}{\partial t}\tau'(0) + \frac{\partial \varphi}{\partial x}\right) 
			= v^\top \left( -\frac{\lambda^\top Cv}{\lambda^\top b}b + Cv \right).
		\end{split}
	\end{equation*}
	Therefore, \eqref{eq:type1_zeno} states that trajectories approach the origin. It remains to show that this convergence occurs in finite time. Indeed, the sum in \eqref{eq:zeno_conditions} is convergent and can be shown by, e.g. the ratio test and l'H\^{o}pital's rule.
\end{proof}
As this proof linearized the flow about the origin, we obtain the following nice corollary about nonlinear systems.
\begin{corollary}
	Consider the planar hybrid system:
	\begin{equation*}
		\begin{cases}
			\dot{x} = f(x), & x\not\in\Sigma, \\
			x^+ = \Delta(x), & x\in\Sigma.
		\end{cases}
	\end{equation*}
	Suppose that $\Sigma\cap\Delta(\Sigma) = \{p\}$, the intersection is transverse, $\Delta(p)=p$, and $f(p)\not\in T_p\Sigma$. Then trajectories with initial condition in a small enough neighborhood of $p$ are Zeno if
	\begin{equation*}
		-1 < v^\top \left[ \Delta_*v - \frac{\lambda^\top \Delta_*v}{\lambda^\top f(p)}f(p)\right] < 1,
	\end{equation*}
	where $v\in \ker\lambda^\top = T_p\Sigma$ is a unit vector and $\Delta_*:T_p\Sigma\to T_p\Sigma$ is the derivative of the reset.
\end{corollary}
\subsection{Type II Zeno}\label{sec:typeII}
Next, consider the affine hybrid system \eqref{eq:affine_hybrid} where
\begin{equation*}
	\Sigma = \left\{ x :
		\lambda^\top x = 0, \
		\nu^\top x < 0
	\right\},
\end{equation*}
and that $\lambda^\top$ \textit{is} a left eigenvector of $C$. The first-order analysis in the previous subsection will no longer suffice. Consider the following example.
\begin{example}[Second-Order Zeno]\label{ex:soZ}
	If $\Sigma$ and its image are no longer transverse, then the trajectory is not reasonably approximated by a piece-wise linear one. Consider the dynamics (which is a model for the bouncing ball):
	\begin{equation*}
		\begin{cases}
			\left. \begin{array}{l}
				\dot{x} = y \\
				\dot{y} = -g
			\end{array}\right\} & (x>0)\vee(x=0 \wedge y>0) \\
			\left. \begin{array}{l}
				x^+ = 0 \\
				y^+ = -ey^-
			\end{array}\right\} & (x=0 \wedge y<0)
		\end{cases}
	\end{equation*}
	where $g>0$ and $0<e<1$, see Fig. \ref{fig:second_order_zeno} for its phase portrait.
	In this case, Zeno always occurs with Zeno time
	\begin{equation*}
		\zeta_2(x_0,y_0) = \frac{1}{g}y_0 + \frac{3}{g(1-e)}\sqrt{y_0^2+2gx_0},
	\end{equation*}
	which is non-linear in the initial conditions, unlike $\zeta_1$.
\end{example}
While the trajectories in Fig. \ref{fig:first_order_zeno} are lines, the trajectories in Fig. \ref{fig:second_order_zeno} are parabolas. As a result, second-order analysis is required to determined whether or not Zeno occurs in these systems. 
\begin{theorem}\label{thm:typeII_Zeno}
	Let $(A,b,C,\lambda,\nu)$ be the data for an affine system with
	\begin{equation*}
		\Sigma = \{ x : \lambda^\top x = 0, \ \nu^\top x < 0\}.
	\end{equation*}
	Suppose that $\lambda^\top$ is a left eigenvector of $C$ and 
	\begin{equation*}
		C\Sigma = \{ x : \lambda^\top x = 0, \ \nu^\top x > 0\}.
	\end{equation*}
	Moreover, let $\lambda^\top b = 0$ and $\lambda^\top Ab\ne 0$. Let $v$ be the unit vector satisfying $\lambda^\top v = 0$ and $\nu^\top v < 0$. Then, for an initial condition $x(0)\ne 0$, the resulting trajectory is Zeno if
	\begin{equation}\label{eq:typeII_zeno}
		-1 < v^\top \left[ Cv - 2\frac{\lambda^\top ACv}{\lambda^\top Ab}b\right] < 1.
	\end{equation}
\end{theorem}
\begin{proof}
	The proof is similar to the proof of Theorem \ref{thm:typeI_Zeno} where the first-order analysis is replaced by second-order. Indeed,
	\begin{equation*}
		\left.\frac{d}{dt}\right|_{t=0} \, \lambda^\top \varphi_t(0) = \lambda^\top b = 0, \quad
		\left.\frac{d^2}{dt^2}\right|_{t=0} \, \lambda^\top \varphi_t(0) = \lambda^\top Ab \ne 0.
	\end{equation*}
	Adapting the steps above results in \eqref{eq:typeII_zeno}.
\end{proof}
\section{Hybrid Optimal Control}\label{sec:hLQR}
The objective of the remainder of this work is to examine optimal control applied to linear \eqref{eq:linear_hybrid} and affine \eqref{eq:affine_hybrid} hybrid dynamical systems by applying the Pontryagin maximum principle. There have been numerous works studying optimal control for hybrid systems, e.g. \cite{minimum_hybrid} along with the discussion and references therein. A hybrid control system will have the form
\begin{equation}\label{eq:hybrid_control_dynamics}
	\begin{cases}
		\dot{x} = f(x,u), & (t,x)\not\in\Sigma, \\
		x^+ = \Delta(x), & (t,x)\in \Sigma,
	\end{cases}
\end{equation}
where $f(x,u)$ is the controlled vector field. In particular, we will be interested in linear and affine hybrid control systems which take the form
\begin{equation*}
	\text{Linear}:\begin{cases}
		\dot{x} = Ax + Bu, & \lambda^\top x \ne 0, \\
		x^+ = Cx, & \lambda^\top x = 0,
	\end{cases} \quad
	\text{Affine}:\begin{cases}
		\dot{x} = Ax + Bu + b, & x\not\in\Sigma, \\
		x^+ = Cx, & x\in\Sigma,
	\end{cases}
\end{equation*}
where $\Sigma$ is an affine guard as laid out in Definition \ref{def:aguard}. The state and control will be elements $x\in\mathbb{R}^n$ and $u\in\mathbb{R}^m$.
For a control problem over the time interval $[t_0,t_f]\subset\mathbb{R}$, admissible controls will be $\mathcal{U} = L^\infty([t_0,t_f], \mathbb{R}^m)$, the set of measurable and essentially bounded functions with values in $\mathbb{R}^m$. For a given $u(\cdot)\in\mathcal{U}$, a solution to \eqref{eq:hybrid_control_dynamics} will be a hybrid arc $\varphi_u:E_u\to\mathbb{R}^n$ such that it is a solution to the time-dependent hybrid system
\begin{equation*}
	\begin{cases}
		\dot{x} = f(x, u(t)), & x\not\in\Sigma, \\
		x^+ = \Delta(x), & x\in\Sigma,
	\end{cases}
\end{equation*}
in the sense of Definition \ref{def:solution}.
\begin{remark}
	For a fixed initial condition, two separate control inputs may generate hybrid time domains that are wildly different. Indeed, at the outset of a problem, it is not clear how many resets will occur or whether or not the resulting arc is exceptional. 
\end{remark}

Suppose that the objective is to minimize the following cost functional (the Bolza problem):
\begin{equation}\label{eq:hybrid_cost}
	J(x_0, u(\cdot)) = \int_{t_0}^{t_f} \, \ell(x(t),u(t)) \, dt + \varphi(x(t_f)),
\end{equation}
where $x(t) = \varphi_u(t,j)$ is the controlled trajectory obeying \eqref{eq:hybrid_control_dynamics}. The index $j\in\mathbb{N}$ is suppressed as the cost is independent of the number of resets experienced. The optimal cost is found by minimizing the above cost functional:
\begin{equation*}
	u_{x_0}(\cdot) = \arg\inf_{u(\cdot)\in\mathcal{U}} \, J(x_0, u(\cdot)).
\end{equation*}
To find necessary conditions for an optimal control, define the Hamiltonian $H:\mathbb{R}^n\times\mathbb{R}^n\to\mathbb{R}$ given by
\begin{equation}\label{eq:opt_hamiltonian}
	H(x,p) = \min_u \left[\langle p, f(x,u)\rangle + p_0\ell(x,u) \right].
\end{equation}
If $x(t)$ is the trajectory arising from the optimal control $u(t)$, then on the continuous components (between resets), the arc satisfies Hamilton's equations,
\begin{equation*}
	\dot{x} = \frac{\partial H}{\partial p}, \quad \dot{p} = -\frac{\partial H}{\partial x},
\end{equation*}
where $p\in\mathbb{R}^n$ is called the \textit{co-state}. The value of the control can then be recovered from \eqref{eq:opt_hamiltonian}. 

Suppose that a reset occurs at $(\bar{t},\bar{x})\in\Sigma$. The state then undergoes the reset $\Delta$. The co-state undergoes the so-called ``Hamiltonian jump conditions'' (cf. e.g. \S7.4.2 in \cite{liberzon} and Theorem 2.2 in \cite{4303244}) and are
\begin{equation}\label{eq:general_jump}
	\begin{split}
		\langle (\Delta_*)^\top p^+ - p^-, \delta x\rangle &= 0, \\
		\langle H(x^-,p^-) - H(x^+,p^+), \delta t\rangle &= 0,
	\end{split}
\end{equation}
where $\Delta_*$ is the Jacobian of the reset map and the variations lie tangent to $\Sigma$, i.e.
\begin{equation*}
	(\delta t,\delta x) \in T_{(\bar{t},\bar{x})}\Sigma.
\end{equation*}
Assuming that the state does not jump at resets, $\Delta = \mathrm{Id}$, these conditions are precisely the famous Weierstrass-Erdmann corner conditions, cf. \S 4.4 in \cite{kirk}. The derivation of \eqref{eq:general_jump} relies on there being a finite number of (more generally isolated) resets, e.g. see Definition 2 in \cite{sussman}. Then this assumption fails (as is the case for exceptional arcs), complications arise, see \cite{zeno_pmp}.

We will focus on two qualitatively distinct cases for the guard:
\begin{enumerate}
	\item[G1.] Temporally-triggered: Impacts occur at specific times, $\overline{t}\in\mathcal{T}$, i.e. $\Sigma = \mathcal{T}\times \mathbb{R}^n$. In this case
	\begin{equation*}
		\delta t = 0, \quad \delta x \in\mathbb{R}^n.
	\end{equation*}
	This is a special case of the third type of affine guard in Definition \ref{def:aguard}.
	\item[G2.] Spatially-triggered: Impacts occur at specific locations. The guard may be either a time-independent affine or half an affine subspace from Definition \ref{def:aguard}. In this case,
	\begin{equation*}
		\delta t \in\mathbb{R}, \quad \lambda^\top\delta x = 0.
	\end{equation*}
\end{enumerate}
We now state the specific problem we wish to study.
\begin{definition}[HAQR]
	The \textit{hybrid affine quadratic regulator} (HAQR) is the minimization of
	\begin{equation}\label{eq:HAQR_cost}
		\begin{split}
			\min_u \, &\frac{1}{2}\int_{t_0}^{t_f} \, \left( x^\top Q x + u^\top Ru + 2x^\top N u\right) \, dt \\ &\hspace{0.75in} + \frac{1}{2}\left[x(t_f)-y\right]^\top F\left[x(t_f)-y\right],
		\end{split}
	\end{equation}
	where the trajectory obeys the controlled dynamics
	\begin{equation}\label{eq:HAQR_dyn}
		\begin{cases}
			\dot{x} = Ax + Bu + b, & (t,x)\not\in\Sigma, \\
			x^+ = Cx, & (t,x)\in\Sigma,
		\end{cases}
	\end{equation}
	where $\Sigma$ is of the form given by (G1) or (G2). The state and control are vectors, $x\in\mathbb{R}^n$, $u\in\mathbb{R}^m$ and the matrices are all of appropriate dimension. Moreover, $Q,F$ are symmetric positive-semidefinite and $R$ is symmetric positive-definite.
	
	The \textit{hybrid linear quadratic regulator} (HLQR) is the a special case of the HAQR with $b = y = 0$ and $\Sigma = \{x:\lambda^\top x = 0\}$.
\end{definition}
The optimal Hamiltonians \eqref{eq:opt_hamiltonian} for these problems are given by
	\begin{equation}\label{eq:Hamiltonians}
		\begin{split}
			H_{\mathrm{LQR}}(x,p) &= \frac{1}{2}x^\top\tilde{Q}x + p^\top \tilde{A}x - \frac{1}{2}p^\top\tilde{R}p, \\
			H_{\mathrm{AQR}}(x,p) &= H_{\mathrm{LQR}} + p^\top b,
		\end{split}
\end{equation}
where the matrices are given by
\begin{equation*}
	\tilde{Q} =  Q - NR^{-1}N^\top, \quad \tilde{A} = A-BR^{-1}N^\top, \quad \tilde{R} = BR^{-1}B^\top.
\end{equation*}
The continuous equations of motion are specified by Hamilton's equations,
\begin{equation}\label{eq:continuous_oc_dyn}
	\begin{split}
		\dot{x} &= \tilde{A} x - \tilde{R} p + b, \\
		\dot{p} &= -\tilde{Q}x - \tilde{A}^\top p,
	\end{split}
\end{equation}
while the jump conditions for the co-state \eqref{eq:general_jump} become
\begin{equation}\label{eq:variational_corner}
	\begin{split}
		\langle C^\top p^+ - p^-, \delta x\rangle &= 0, \\
		\langle H_{\mathrm{\square QR}}(x^-,p^-) - H_{\mathrm{\square QR}}(x^+,p^+), \delta t\rangle &= 0,
	\end{split}
\end{equation}
with $\square\in \{L,A\}$. For the remainder, we assume that we are exclusively in the affine quadratic regulator unless otherwise stated. As the linear quadratic regulator is a special case, there is no loss of generality.
\subsection{Temporally-Triggered Resets}
Here, we focus on the HAQR when resets occur at specified times (G1). Let $\mathcal{T} =\{t_i\}\subset\mathbb{R}$ be a uniformly separated discrete subset, i.e.
\begin{equation*}
	\inf_{i\ne j} \left| t_i - t_j \right| = \delta > 0,
\end{equation*}
and let $\Sigma = \mathcal{T}\times\mathbb{R}^n$. 
The case when $\mathcal{T} = \kappa \mathbb{Z}$ is periodic is the object of study in \cite{periodic_hLQR} and covers much of the contents of this section in the linear (non-affine) case. The affine case appears to be new although it is a straight-forward extension. For the purposes here, $\mathcal{T}$ need not be periodic.

The resulting temporally triggered affine/linear hybrid dynamics are 
\begin{equation}\label{eq:tthLQR}
	\begin{cases}
		\dot{x}(t) = Ax(t) + Bu(t)+b, & t\not\in\mathcal{T} \\
		x(t^+) = Cx(t), & t\in\mathcal{T}.
	\end{cases}
\end{equation}
We will assume throughout that the above system is stabilizable as this will allow for solvability of a Riccati equation in Theorem \ref{thm:temp_aqr} below. In the linear periodic case with $\mathcal{T} = \kappa\mathbb{Z}$, this takes the form \cite{stable_hybrid}
	\begin{equation}\label{eq:stabilizable_temporal}
		\mathrm{rank} \left[ Ce^{A\kappa}-s\cdot\mathrm{Id}, B, AB, A^2B, \ldots, A^{n-1}B\right] = n.
	\end{equation}
for $s\in\mathbb{C}$ with $|s|\geq 1$.

When $t\not\in\mathcal{T}$, the co-states evolve according to the continuous problem \eqref{eq:continuous_oc_dyn}. At the moment of impact, the co-states jump according to the variational conditions \eqref{eq:variational_corner} subject to $\delta t = 0$ and $\delta x\in\mathbb{R}^n$ being free. Therefore, the jump is given by
\begin{equation*}
	C^\top p^+ = p^-.
\end{equation*}
While the value function for classical LQR is quadratic, the value function for AQR is affine
	\begin{equation*}
		V(t,x) = \frac{1}{2}x^\top S(t)x + c(t)^\top x.
	\end{equation*}
	This motivates the affine ansatz $p(t) = S(t)x(t) + c(t)$.
Applying this yields
\begin{equation*}
	\begin{split}
		C^\top p^+ &= C^\top \left( S^+x^+ + c^+\right) \\
		&= C^\top S^+ Cx^- + C^\top c^+,
	\end{split}
\end{equation*}
while the co-state jump is
\begin{equation*}
	C^\top p^+ = p^- = S^-x^- + c^-.
\end{equation*}
Combining these returns the jump map
\begin{equation}\label{eq:temp_aqr_jump}
	\begin{split}
		S^- &= C^\top S^+ C, \\
		c^- &= C^\top c^+.
	\end{split}
\end{equation}
Optimal trajectories for the temporally triggered HAQR are synthesized by solving
\begin{equation}\label{eq:tthAQR_costate}
	\begin{cases}
		\begin{array}{l}
			\dot{S} = -\tilde{A}^\top S - S\tilde{A} + S\tilde{R}S - \tilde{Q} \\
			\dot{c} = \left[ -\tilde{A}^\top + S(t)\tilde{R}\right]c - S(t)b
		\end{array} & t\not\in\mathcal{T}, \\[3ex]
		\begin{array}{l}
			S^- = C^\top S^+ C \\
			c^- = C^\top c^+
		\end{array} & t\in\mathcal{T}.
	\end{cases}
\end{equation}
\textit{backwards} with the terminal conditions
	\begin{equation*}
		S(t_f) = F, \quad c(t_f) = -2Fy.
\end{equation*}
The \textit{forward} dynamics are given by
\begin{equation}\label{eq:tthAQR_state}
	\begin{cases}
		\dot{x} = \tilde{A}x - \tilde{R}S(t)x - \tilde{R}c(t) + b, & t\in\mathcal{T}, \\
		x^+ = Cx, & t\in\mathcal{T}.
	\end{cases}
\end{equation}
\begin{remark}
	In \eqref{eq:tthAQR_state}, the $x$-dynamics are solved forward while the $(S,c)$-dynamics are backwards in \eqref{eq:tthAQR_costate}. As such, the jump prescribes $x^+$ and $(S^-,c^-)$. As the $(S,c)$-dynamics are backwards, there is no immediate issue if the jumping matrix, $C$, is degenerate. Naturally, controllability issues may arise in this case, but this is reserved for future study.
\end{remark}
\begin{remark}
	It is important to notice that the optimal control problem for temporally triggered jumps allows for the decoupling of the forward and backward dynamics. This will not be the case for spatially triggered jumps.
\end{remark}
\begin{remark}
	The temporally triggered HAQR always admits unique solutions for the costates \eqref{eq:temp_aqr_jump} \textit{backwards} but not necessarily forwards. Additionally, the feedback control is affine,
	\begin{equation*}
		u^* = -R^{-1}\left( N^\top + B^\top S(t)\right)x - R^{-1}B^\top c(t),
	\end{equation*}
	and becomes linear for HLQR.
\end{remark}

If $\mathcal{T} = \kappa\mathbb{Z}$, the corresponding steady-state periodic Riccati equation is
\begin{gather}
	\dot{S} = -\tilde{A}^\top S - S\tilde{A} + S\tilde{R}S - \tilde{Q}, \quad S(\kappa) = C^\top S(0) C, \label{eq:periodic_S}\\
	\dot{c} = \left[ -\tilde{A}^\top + S(t)\tilde{R}\right]c - S(t)b, \quad c(\kappa) = C^\top c(0). \label{eq:periodic_c}
\end{gather}
Solvability of \eqref{eq:periodic_S} is studied in \cite{periodic_hLQR} which requires that the system is stabilizable \eqref{eq:stabilizable_temporal}.  Below, solvability of \eqref{eq:periodic_c} is included.
\begin{theorem}\label{thm:temp_aqr}
	Consider the matrices
	\begin{equation*}
		\renewcommand*{\arraystretch}{1.5}
		Z = \left[\begin{array}{c;{2pt/2pt}c}
			\tilde{A} & -\tilde{R} \\ \hdashline[2pt/2pt]
			-\tilde{Q} & -\tilde{A}^\top
		\end{array}\right], \quad 
		\left[\begin{array}{c;{2pt/2pt}c}
			P_{11} & P_{12} \\ \hdashline[2pt/2pt]
			P_{21} & P_{22} \end{array}\right] = e^{Z\kappa}.
	\end{equation*}
	Then solutions to \eqref{eq:periodic_S} correspond to solutions to
	\begin{equation}\label{eq:triggered_riccati}
		C^\top S_0 C \left[ P_{11} + P_{12}S_0\right] = P_{21} + P_{22}S_0.
	\end{equation}
	Let $\Phi(t,s)$ be the fundamental matrix solution
	\begin{equation*}
		\frac{\partial}{\partial t}\Phi(t,s) = \left[-\tilde{A}^\top + S(t)\tilde{R}\right] \Phi(t,s), \quad \Phi(s,s) = \mathrm{Id}.
	\end{equation*}
	Then solutions to \eqref{eq:periodic_c} can be found by
	\begin{equation}\label{eq:triggered_bias}
		c_0 = \left[ \Phi(\kappa,0) - C^\top\right]^{-1}\cdot \int_0^\kappa \, \Phi(\kappa,\tau)S(\tau)b\, d\tau,
	\end{equation}
	assuming that $\Phi(\kappa,0) - C^\top$ is invertible.
	
	Moreover, if \eqref{eq:stabilizable_temporal} holds, then \eqref{eq:triggered_riccati} admits a solution.
\end{theorem}
\begin{proof}
	We begin with \eqref{eq:triggered_riccati}. A solution $S(t)$ corresponds to a solution of $\dot{z} = Zz$ where $z = [x,p]^\top$ with $p = Sx$. The solutions to this linear system are
	\begin{equation*}
		\begin{split}
			x_\kappa &= P_{11}x_0 + P_{12}p_0, \\
			p_\kappa &= P_{21}x_0 + P_{22}p_0.
		\end{split}
	\end{equation*}
	Using the relation $p = Sx$, the above is
	\begin{equation*}
		S_\kappa \left[ P_{11} + P_{12}S_0\right] x_0 = \left[ P_{21} + P_{22}S_0\right] x_0.
	\end{equation*}
	The result follows from the jump map $S^- = C^\top S^+C$ with $S^- = S_\kappa$ and $S^+ = S_0$.	It is shown in \cite{periodic_hLQR} that if \eqref{eq:stabilizable_temporal} holds, then \eqref{eq:triggered_riccati} admits a solution. 
	Equation \eqref{eq:triggered_bias} can be found via the standard formula for non-homogeneous linear systems, e.g. \S1.10 in \cite{perko}
	\begin{equation*}
		c_\kappa = \Phi(\kappa,0)c_0 - \int_0^\kappa \, \Phi(\kappa,\tau) S(\tau)b \, d\tau,
	\end{equation*}
	and applying the reset $c_\kappa = C^\top c_0$.
\end{proof}
\subsection{Spatially-Triggered Resets}
Suppose now that the guard is of the form (G2). That is, the guard is time-independent and state-dependent. The resulting spatially triggered linear hybrid dynamics are
\begin{equation}\label{eq:st_control}
	\begin{cases}
		\dot{x} = Ax + Bu + b, & x\not\in\Sigma, \\
		x^+ = Cx, & x\in\Sigma.
	\end{cases}
\end{equation}
The optimal control problem for \eqref{eq:st_control} appears to be largely unstudied as most results on this topic deal with temporally-triggered resets. While temporally-triggered resets preserve the linear (or affine) structure, this is no longer true for spatially-triggered systems.

In-between resets, as with the temporally-triggered case, the states/co-states evolve according to \eqref{eq:continuous_oc_dyn}. At the moment of reset, the co-sates jump according to the variational conditions \eqref{eq:variational_corner} subject to $\delta t\in\mathbb{R}$ being free and $\lambda^\top\delta x = 0$. This results in the jump conditions
\begin{equation}\label{eq:space_jump}
	\begin{split}
		p^- &= C^\top p^+ + \varepsilon\cdot\lambda, \\
		H_{\mathrm{AQR}}^+ &= H^-_{\mathrm{AQR}},
	\end{split}
\end{equation}
where the multiplier $\varepsilon$ is chosen to enforce energy conservation.

\begin{remark}\label{rmk:non_beating}
	The co-state reset conditions \eqref{eq:space_jump} implicitly assume that a single reset occurs, i.e. $Cx\not\in\Sigma$. As such the analysis presented in this section assumes that the states does not lie within the first beating set, $\Sigma_1$. Section \ref{sec:beat} focuses on the case when beating inevitability occurs.
\end{remark}
Let $x := x^-$ and $p := p^+$ be the state pre-reset and co-state post-reset. Energy conservation takes the form
	\begin{equation*}
		H_{\mathrm{AQR}}(x,C^\top p + \varepsilon\lambda) = H_{\mathrm{AQR}}(Cx,p).
\end{equation*}
As the Hamiltonian \eqref{eq:Hamiltonians} is quadratic in $p$, the multiplier must solve the quadratic equation
\begin{equation*}
	\alpha \varepsilon^2 + \beta\varepsilon + \gamma = 0,
\end{equation*}
where the coefficients are given by
\begin{equation}\label{eq:LQR_coeff}
	\begin{split}
		\alpha &= -\frac{1}{2}\lambda^\top \tilde{R} \lambda, \\
		\beta(x,p) &= \lambda^\top \left[ \tilde{A} x - \tilde{R} C^\top p \right] + \lambda^\top b, \\
		\gamma(x,p) &= H_{AQR}(x,C^\top p) - H_{AQR}(Cx,p).
	\end{split}
\end{equation}
An immediate cause of concern is the fact that unique solutions need not exist for \eqref{eq:space_jump} as quadratic equations may have 0, 1, or 2 roots based on the value of the discriminant 
\begin{equation*}
	\mathcal{D}(x,p) = \beta(x,p)^2 - 4\alpha\gamma(x,p).
\end{equation*}
For the moment, we shall assume that there exists a unique solution to \eqref{eq:space_jump} and relegate this issue to \S\ref{sec:complications}. 
Optimal trajectories for the spatially triggered AQR follow the continuous dynamics \eqref{eq:continuous_oc_dyn} when $x\not\in\Sigma$ and reset according to $x^+ = Cx$ and \eqref{eq:space_jump} when $x\in\Sigma$ with the terminal conditions
	\begin{equation*}
		x(t_0) = x_0, \quad p(t_f) = Fx(t_f) - 2Fy.
\end{equation*}
It would be desirable to decouple the forward and backward dynamics by developing a Riccati equation for the co-states as was done by \eqref{eq:tthAQR_costate} for the temporally-triggered case. Unfortunately, this is not possible as the occurrence of resets now depends on the state.

Suppose that a stationary solution were to exist of the form $p = Sx$. The corresponding spatially-triggered algebraic Riccati equation (for the HLQR) is
	\begin{equation}\label{eq:spatial_riccati}
		\begin{split}
			0 &= -\tilde{A}^\top S - S\tilde{A} + S\tilde{R}S - \tilde{Q}, \\
			Sx &= C^\top SCx + \varepsilon(x,SCx)\cdot\lambda, \quad x\in\Sigma = 0.
		\end{split}
	\end{equation}
	Solvability of this is not straight forward as the right side of the second condition is not necessarily linear in $x$, see Fig. \ref{fig:weak_control} in \S\ref{sec:spat_trig} below.
\section{Complications with the Co-States}\label{sec:complications}
We return to the solvability issue of \eqref{eq:space_jump} by considering the two qualitatively distinct cases of weakly and strongly actuated resets. Additionally, the cases of beating and Zeno are studied as the jump conditions in these cases do not obey \eqref{eq:space_jump}. This analysis does not seem to have been done before.
\subsection{Weakly Actuated Resets}
Notice that among the three coefficients in \eqref{eq:LQR_coeff}, $\alpha$ is constant while both $\beta$ and $\gamma$ depend on $(x,p)$. A way to obtain a unique solution to \eqref{eq:LQR_coeff} is to have $\alpha = 0$ which is a property of the system and holds independent of the state/co-state. It turns out that the vanishing of this number is related to the actuation of the system.
This condition will be first stated for more general control systems (similar to what was introduced in the note \cite{clarkopreagraven}) before being specialized to affine hybrid systems.

\begin{definition}[Weakly Actuated Resets]\label{def:war}
	Let $M$ be a $n$-dimensional manifold and $\Sigma\subset M$ be a co-dimension 1 embedded submanifold. Consider the controlled hybrid dynamics
	\begin{equation*}
		\begin{cases}
			\dot{x} = f(x) + \sum_i \, g_i(x)u^i, & x\not\in\Sigma, \\
			x^+ = \Delta(x), & x\in\Sigma.
		\end{cases}
	\end{equation*}
	This system has \textit{weakly actuated resets} (WAR) if the control vector fields are all tangent to $\Sigma$, i.e.
	\begin{equation}\label{eq:war}
		g_i(x) \in T_x\Sigma, \quad \forall i \ \text{and} \ \forall x\in\Sigma.
	\end{equation}
\end{definition}
\begin{proposition}
	An affine hybrid control system with affine guard of type (G2) has weakly actuated resets if $\lambda^\top B = 0$.
\end{proposition}
Weakly actuated resets for an affine hybrid control system is useful as it reduces \eqref{eq:space_jump} from a quadratic to a linear equation.
\begin{proposition}\label{prop:war}
	If a system has WAR, then the coefficients in \eqref{eq:LQR_coeff} reduce to
	$\alpha=0$ and $\beta = \lambda^\top Ax + \lambda^\top b$.
\end{proposition}
Unfortunately, divide-by-zero issues arise in the linear case if
\begin{equation*}
	x\in \ker \lambda^\top \cap \ker \lambda^\top A.
\end{equation*}
This intersection is always nontrivial whenever the dimension of the ambient space $n\geq 3$, as both of these sets are hyperplanes. 
Recall, from Section \ref{sec:LHDS}, that this intersection is precisely the invariant guard, $\Sigma^A$. 
Taking Remark \ref{rmk:non_beating} into account, we have the following.
\begin{proposition}
	Suppose that \eqref{eq:st_control} has WAR and $b=0$. Then a unique solution exists to \eqref{eq:space_jump} for any $x\in\Sigma\setminus \left( \Sigma^A \cup \Sigma_1\right)$.
\end{proposition}

\begin{remark}
	The reason why \eqref{eq:war} is referred to as a weakly actuated reset is that the controls have no direct (first-order) influence on resets. For the nonlinear case, suppose that $\Sigma = h^{-1}(0)$ can be described by a regular level set. Then,
	\begin{equation*}
		\frac{d}{dt} h(x) = dh_x \left( f(x) + \sum_i \, g_i(x)u^i \right) = dh_x(f(x)),
	\end{equation*}
	as $dh(g_i) = 0$. The derivative in $h$ is completely independent of the controls, i.e.
	\begin{equation*}
		\frac{\partial}{\partial u} \frac{d}{dt}h(x) = 0.
	\end{equation*}
\end{remark}
\begin{example}[Mechanical Systems]
	Although the WAR condition, \eqref{eq:war}, appears to be quite special, mechanical systems always satisfy this. Let $x = (q,v)$ be the position and velocity of a system. The equations of motion (linearized) are given by
	\begin{equation*}
		\begin{split}
			\dot{q} &= v, \\
			\dot{v} &= Vq + Kv + \tilde{B}u +\tilde{b}.
		\end{split}
	\end{equation*}
	As such, the controlled dynamics has the form
	\begin{equation*}
		\begin{bmatrix}
			\dot{q} \\ \dot{v}
		\end{bmatrix} = \begin{bmatrix}
			0 & \mathrm{Id} \\ V & K
		\end{bmatrix}\begin{bmatrix} q \\ v \end{bmatrix} + \begin{bmatrix}
			0 \\ \tilde{B}
		\end{bmatrix}u + \begin{bmatrix} 0 \\ \tilde{b} \end{bmatrix}.
	\end{equation*}
	The controls only have direct influence over the velocities (not positions) and impacts are triggered only by positions (not velocities). As such, $\Sigma$ is generated by a vector of the form
	\begin{equation*}
		\lambda = \begin{bmatrix}
			\tilde{\lambda}^\top & 0
		\end{bmatrix}^\top.
	\end{equation*}
	It is clear to see that
	\begin{equation*}
		\lambda^\top B = 
		\begin{bmatrix}
			\tilde{\lambda}^\top & 0
		\end{bmatrix}\begin{bmatrix}
			0 \\ \tilde{B}
		\end{bmatrix} = 0.
	\end{equation*}
	Therefore, mechanical impact systems have WAR. Additionally, as only velocities jump at resets the matrix $C$ has the form
		\begin{equation*}
			C = \begin{bmatrix} \mathrm{Id} & 0 \\ C_1 & C_2 \end{bmatrix}.
		\end{equation*}
		If the resulting dynamics are Zeno, then they must be Type II (see Section \ref{sec:typeII}). As such, (nonlinear) \textit{mechanical systems provide a class of spatially triggered hybrid systems with weakly actuated resets and Type II Zeno}.
\end{example}
\subsection{Strongly Actuated Resets}\label{sec:sar}
A system is said to have strongly actuated resets if it does not have weakly actuated resets, i.e. $\alpha \ne 0$. In this case, \eqref{eq:space_jump} remains quadratic. It turns out that the two solutions characterize which side of the guard the trajectory intersects.
\begin{proposition}\label{prop:strong_solutions}
	Let $\mathcal{D} = \beta^2-4\alpha\gamma$ be the discriminant with coefficients from \eqref{eq:LQR_coeff}. Then $|\lambda^\top \dot{x}^-| = |\sqrt{\mathcal{D}}|$.
\end{proposition}
\begin{proof}
	The multiplier, $\varepsilon$, is given by the quadratic equation
	\begin{equation*}
		\varepsilon_{\pm} = \frac{-\beta \pm \sqrt{\mathcal{D}}}{2\alpha}.
	\end{equation*}
	Immediately before the reset (with $x=x^-$ and $p=p^+$),
	\begin{equation*}
			\begin{split}
				\lambda^\top\dot{x}^- &= \lambda^\top \left[ \tilde{A}x^- - \tilde{R} p^- + b\right] \\
				&= \lambda^\top \tilde{A}x - \lambda^\top\tilde{R}\left[ C^\top p +\varepsilon\cdot\lambda\right] +\lambda^\top b\\
				&= \lambda^\top\tilde{A}x - \lambda^\top\tilde{R}C^\top p - \varepsilon\lambda^\top\tilde{R}\lambda + \lambda^\top b \\
				&= \beta + 2\varepsilon\alpha.
			\end{split}
		\end{equation*}
		Therefore $\lambda^\top \dot{x}^- = \sqrt{\mathcal{D}}$ or $-\sqrt{\mathcal{D}}$.
\end{proof}
When two solutions for $\varepsilon$ exist, there exists a reasonable interpretation to this dichotomy based on the direction the arc intersects the guard. Suppose that the affine guard has the form
\begin{equation*}
	\Sigma = \{x : \lambda^\top x = a\}.
\end{equation*}
Decompose the state-space as 
\begin{equation*}
	\mathbb{R}^n = M^+ \cup M^- \cup \Sigma, \quad \begin{array}{c}
		M^+ = \{x : \lambda^\top x > a\}, \\
		M^- = \{x : \lambda^\top x < a\}.
		\end{array}
\end{equation*}
Then, the correct solution to \eqref{eq:space_jump} is determined by which space the trajectory approaches the guard in: suppose that $x(t^*)\in\Sigma$, then for $\delta>0$ small enough
\begin{equation*}
	\begin{split}
		x(t^*-\delta) \in M^+ &\implies \lambda^\top\dot{x}^- < 0 \implies \varepsilon = \frac{-\beta - \sqrt{\mathcal{D}}}{2\alpha},\\
		x(t^*-\delta) \in M^- &\implies \lambda^\top\dot{x}^- > 0 \implies \varepsilon = \frac{-\beta + \sqrt{\mathcal{D}}}{2\alpha}.
	\end{split}
\end{equation*}
A visualization of this case is shown in Fig. \ref{fig:two_solutions}.
\begin{figure}
	\centering
	\begin{tikzpicture}
		\draw [<->, thick, blue] (-1,0) -- (5,0);
		\draw [<->, thick, red] (0,-2.5) -- (0,2.5);
		\node [below right] at (5,0) {$\textcolor{blue}{\Sigma}$};
		\node [above left] at (0,2.5) {$\textcolor{red}{\Delta(\Sigma)}$};
		\node at (4.5,2) {$M^+$};
		\node at (4.5,-2) {\textcolor{violet}{$M^-$}};
		\draw[thick, ->] (1.5,1) to [out=20, in=80] (3,0);
		\draw[thick, dashed, ->, violet] (1.5,-1) to [out=-20, in=280] (3,0);
		\node [above right] at (3,0) {$x_1^-$};
		\node [below right] at (3,0) {\textcolor{violet}{$x_2^-$}};
		\draw[thick, ->] (0, 2) to [out=20, in=100] (3.5,1.5);
		\node[below right] at (0,2) {$x^+$};
	\end{tikzpicture}
	\caption{A schematic for the case of two solutions for the co-state jump. In this case, as $x^+\in M^+$, the correct solution yields the curve $x_1$.}
	\label{fig:two_solutions}
\end{figure}
A interpretation on why strongly actuated resets have two solutions while weakly actuated only have one is that $\lambda^\top \dot{x}$ can be influenced by the controls in the latter and not in the former. As a result, both trajectories (the solid and dashed arcs in Fig. \ref{fig:two_solutions}) are possible in systems with strongly actuated resets while only one is possible in systems with weakly actuated resets.
\subsection{Beating}\label{sec:beat}
The entirety of the above analysis has tacitly assumed that $x^-\in\Sigma\setminus\Sigma_1$, i.e. no beating occurs at the point of impact. As shown in Proposition \ref{prop:size_beating}, although smaller, the first-beating set is non-trivial. Consider the case when $x\in\Sigma_1\setminus \Sigma_2$ and $\lambda^\top$ is \textit{not} a left eigenvector of $C$ (the other cases can be determined by a similar analysis).

Let the first-beating set be represented by
\begin{equation*}
	\Sigma_1 = \left\{ x\in\mathbb{R}^n : \lambda^\top x = \mu^\top x = 0\right\},
\end{equation*}
where $\mu^\top = \lambda^\top C$. The variational conditions \eqref{eq:variational_corner} become
\begin{equation*}
	\begin{split}
		p^- &= \left(C^2\right)^\top p^+ + \varepsilon\cdot\lambda + \eta\cdot\mu, \\
		H_{\mathrm{AQR}}^+ &= H_{\mathrm{AQR}}^-.
	\end{split}
\end{equation*}
This system is under-determined as there are now two unknown multipliers, $\varepsilon$ and $\eta$. To resolve this issue, a technique similar to that of \cite{submersive_resets} is utilized. Assuming sufficient regularity, denote the 1-dimensional manifold (parameterized by $\lambda$ and $\mu$) by
\begin{equation*}
	\Pi_1 := \left\{ \left(C^2\right)^\top p^+ + \varepsilon\cdot\lambda + \eta\cdot\mu : H_{\mathrm{AQR}}^+ = H_{\mathrm{AQR}}^-\right\}.
\end{equation*}
Consider, without loss of generality, the first arc of the trajectory. This arc must satisfy the following boundary conditions
\begin{equation*}
	\begin{array}{ll}
		x(0) = x_0, & x(t_1) \in\Sigma_1, \\
		p(0) \in \mathbb{R}^n, & p(t_1) \in \Pi_1.
	\end{array}
\end{equation*}
These have $(2n+1)$ unknowns given by $x(0)$, $p(0)$, and $t_1$. There are now $(2n+1)$ relations arising from
\begin{equation*}
	\begin{array}{ll}
		x(0) = x_0, & \text{$n$ relations}, \\
		x(t_1)\in\Sigma_1, & \text{2 relations}, \\
		p(t_1) \in \Pi_1, & \text{$(n-1)$ relations}.
	\end{array}
\end{equation*}
As the number of unknowns and relations are now equal, this problem is no longer under-determined. If, however, multiple solutions exist, a technique similar to finding the unique solution for strongly actuated resets in \S\ref{sec:sar} can be used.

This procedure can be extended to the $k^{th}$-beating set. Suppose that the system is trivially blocking and denote $\lambda_j^\top := \lambda^\top C^{j}$. Then,
\begin{equation*}
	\Sigma_k = \left\{ x : \lambda_0^\top x = \lambda_1^\top x = \ldots = \lambda_k^\top x = 0\right\}.
\end{equation*}
Define the $k$-dimensional manifold (assuming sufficient regularity),
\begin{equation*}
	\Pi_k := \left\{ \left(C^{k+1}\right)^\top p^+ + \sum_{j=0}^k \, \eta_j\cdot\lambda_j : H_{\mathrm{AQR}}^+ = H_{\mathrm{AQR}}^-\right\}.
\end{equation*}
Then an arc that terminates on $\Sigma_k$ must satisfy the boundary conditions
\begin{equation}\label{eq:kbeating_boundary}
	\begin{array}{ll}
		x(0) = x_0, & x(t_1) \in\Sigma_k, \\
		p(0) \in \mathbb{R}^n, & p(t_1) \in \Pi_k.
	\end{array}
\end{equation}
In principle, optimal trajectories that may lie within the beating sets can be found by solving \eqref{eq:kbeating_boundary}. However, this is highly non-trivial for two reasons.
\begin{enumerate}
	\item It is not immediately clear which beating set the impact should occur in.
	\item Even if the $k$ is known, the set $\Pi_k$ also depends on the value of the co-state after impact. 
\end{enumerate}
Developing a numerical algorithm to deal with beating solutions is beyond the scope of this paper and is a topic for future work.
\subsection{Zeno}
To this point, necessary conditions for optimality for non-exceptional hybrid arcs have been studied. However, the existence problem has not been addressed. It turns out that solutions need not exist. Notably, systems whose uncontrolled dynamics are Zeno (cf. Theorems \ref{thm:typeI_Zeno} and \ref{thm:typeII_Zeno}) will fail to have non-Zeno optimal arcs over long enough time horizons. For an explicit demonstration of the following to the special case of the bouncing ball (an affine hybrid system with Type II Zeno trajectories) see \cite{zeno_pmp} and the numerical example below in \S\ref{subsec:Zeno_ex}.

\begin{definition}[Robustly Zeno]
	An affine hybrid control system is \textit{robustly Zeno} if there exists $\delta>0$ such that if $\lVert u\rVert_R = u^\top R u < \delta$, then the resulting controlled trajectories are all Zeno.
\end{definition}
The notion of robustly Zeno follows from its structural stability \cite{hybrid_sstable}. 

\begin{proposition}\label{prop:zeno_demanding}
	Consider a robustly Zeno affine hybrid control system. Consider the modified cost functional 
	\begin{equation*}
		J_\varepsilon(x_0,t_f) = \inf_{\substack{u(\cdot)\in\mathcal{U} \\ \lVert x(\cdot)\rVert > \varepsilon}} \, J(x_0, u(\cdot)),
	\end{equation*}
	i.e. there is an additional constraint that the trajectory must always be $\varepsilon$-far from the Zeno point. Then $J_\varepsilon(x_0,t_f)\to \infty$ as $t_f\to\infty$.
\end{proposition}

To compare these ``uniformly non-Zeno'' trajectories with Zeno ones, let us construct a Zeno arc in the following fashion. Consider the HAQR problem \eqref{eq:HAQR_cost} and \eqref{eq:HAQR_dyn} and suppose that all trajectories of the uncontrolled system are Zeno (see Theorems \ref{thm:typeI_Zeno} and \ref{thm:typeII_Zeno}). For simplicity, assume that $Q=0$, $N=0$, and $t_0=0$.
For an initial condition, let $\zeta(x_0)$ be the Zeno time and suppose that $t_f>>\zeta(x_0)$. Define the control function
\begin{equation}\label{eq:zeno_control}
	u_{\text{Zeno}}(t) = \begin{cases}
		0, & t<s^* \\
		u^*(t), & t\geq s^*
	\end{cases}
\end{equation}
where $s^*$ and $u^*$ solve the optimal control problem
\begin{gather*}
	\min_{u,s} \, \int_{s}^{t_f} \, \frac{1}{2}u^\top R u\, dt + \frac{1}{2}\lVert x(t_f)-y\rVert_F^2, \\
	\text{subject to \eqref{eq:HAQR_dyn}, and } x(s) = 0,
\end{gather*}
where $\lVert x\rVert_F := x^\top F x$.
The resulting trajectory is Zeno when $s^*>\zeta(x_0)$ and let the resulting cost be expressed as
\begin{equation*}
	J_{\text{Zeno}}(x_0,t_f) = \int_{s^*}^{t_f} \, \frac{1}{2}u_{\text{Zeno}}^\top R u_{\text{Zeno}} \, dt + \frac{1}{2}\lVert x(t_f) - y\rVert_F^2.
\end{equation*}
It is important to notice that this cost is independent from $t_f$. This fact implies that this Zeno control scheme will eventually out-perform non-Zeno strategies. 
\begin{corollary}\label{cor:Zeno_control}
	Suppose that the HAQR problem is robustly Zeno. Then for any $x_0$ and all $\varepsilon>0$, there exists $t_f$ such that
	\begin{equation*}
		J_\varepsilon(x_0,t_f) > J_{\text{Zeno}}(x_0,t_f).
	\end{equation*}
\end{corollary}
A question for future work is whether or not $J_\varepsilon(x_0,t_f)\to J_{\text{Zeno}}(x_0,t_f)$ as $\varepsilon\to 0$ for $t_f$ fixed.

\section{Examples}\label{sec:examples}
The code for the results in this section is written in Matlab and is publicly available at \url{https://github.com/wiclark/hybrid_LQR/}. Four numerical examples are provided below. 
	\begin{enumerate}
		\item Section \ref{sec:ex_temp} exhibits a temporally-triggered linear hybrid system and solves the corresponding periodic Riccati equation \eqref{eq:triggered_riccati}.
		\item Section \ref{sec:spat_trig} presents a simple example for both weakly actuated and strongly actuated spatially-triggered systems. 
		\item Section \ref{sec:dmss} examines the mechanical example of the double spring system with impacts as an example of a 4-dimensional spatially-triggered affine system.
		\item Section \ref{subsec:Zeno_ex} uses dynamic programming to find Zeno trajectories in accordance with Corollary \ref{cor:Zeno_control}.
\end{enumerate}
\subsection{Temporally-Triggered}\label{sec:ex_temp}
Consider the model of legged locomotion with foot-slip studied in \cite{foot_slip} with an added control on the $\xi$-dynamics:
\begin{equation*}
	\begin{split}
		\dot{\theta} &= \varepsilon\xi, \\
		\dot{\xi} &= C\left[ \sin\theta \left(\cos\alpha + \varepsilon\xi^2\cos\theta\right) - \cos\theta \left(\eta + \xi\cos\theta\right)\right] + u, \\
		\dot{\eta} &= \sin\alpha - \eta - \xi\cos\theta,
	\end{split}
\end{equation*}
where $C = \mu/(1+\mu\sin^2\theta)$. Here, $\theta$ corresponds to the angle of the leg with the ground, $\xi$ the leg's angular velocity, and $\eta$ the sliding velocity of the foot. The guard is $\Sigma = \{\theta = -\delta\}$ and reset
\begin{equation*}
	\Delta(\theta,\xi,\eta) = (\theta+2\delta, \cos(2\delta)\xi, \eta+\cos\delta[1-\cos(2\delta)]\xi),
\end{equation*}
where $(\mu,\varepsilon,\delta)$ are model parameters controlling the geometry of the leg, the strength of the friction, and the step width, respectively.

Using the parameters
\begin{equation*}
	\alpha = \frac{\pi}{16}, \quad \delta = \frac{\pi}{8}, \quad \varepsilon = 2, \quad \mu = 2,
\end{equation*}
a periodic orbit was found with period $\kappa = 1.3325$. Linearizing about this trajectory, and applying the weight matrices,
\begin{equation*}
	Q = \mathrm{Id}_{3\times 3}, \quad N = 0, \quad R = 1,
\end{equation*}
a positive-definite symmetric solution to \eqref{eq:triggered_riccati} is found to be
\begin{equation*}
	S_0 = \begin{bmatrix}
		1.2354 & 1.3802 & 0.2103 \\
		1.3802 & 3.6486 & 0.1987 \\
		0.2103 & 0.1987 & 0.5021
	\end{bmatrix}.
\end{equation*}
\subsection{Spatially Triggered}\label{sec:spat_trig}
We next consider the case of the spatially-triggered HLQR. We present two examples, one where the resets are strongly actuated and another where they are weakly actuated. The values for the data are largely arbitrary.

The purpose of these examples is to numerically validate the Hamiltonian jump condition \eqref{eq:space_jump}. If the optimal arc never resets, then the cost function can be cast as a function on the initial co-states
\begin{equation*}
	J(p_0) = \int_{t_0}^{t_f} \, \ell\left( x(t), u\left( x(t),p(t)\right) \right) \, dt + \varphi\left(x(t_f)\right),
\end{equation*}
where $(x(t),p(t))$ follow the dynamics \eqref{eq:continuous_oc_dyn}. Then the following holds
\begin{equation*}
	\arg\min_{p_0} \, J(p_0) \implies p(t_f) = Fx(t_f).
\end{equation*}
Likewise, if an optimal arc has a single reset, the cost can be viewed as
\begin{equation*}
	J^* = \min_{p_0,p_1^+} \, J(p_0,p_1^+),
\end{equation*}
where $p_0$ is the initial co-state and $p_1^+$ is the co-state immediately after the reset, see Fig. \ref{fig:spatially_triggered}. This leads to 
\begin{equation*}
	\arg\min_{p_0,p_1^+} \, J(p_0,p_1^+) \implies \begin{cases}
		p(t_f) = Fx(t_f), \\
		p_1^- = C^\top p_1^+ + \varepsilon\cdot\lambda, \\
		H_{LQR}^+ = H_{LQR}^-.
	\end{cases}
\end{equation*}
We numerically verify this relationship by computing the minimization via Matlab's \texttt{fminsearch} and checking that the jump conditions hold.
\begin{figure}
	\centering
	\begin{subfigure}[t]{0.49\textwidth}
		\centering
		\begin{tikzpicture}[scale=0.7]
			\draw[<->, thick, blue] (-1,0) -- (5,0);
			\draw[<->, thick, red] (0,-1) -- (0,5);
			\node [above right, blue] at (5,0) {$\mathrm{ker}\lambda^\top$};
			\node [above left, red] at (0,5) {$C \mathrm{ker}\lambda^\top$};
			\draw[thick] (3, 2) to [out=-20, in=60] (3.5,0);
			\node [above right] at (3,2) {$(x_0,p_0)$};
			\draw[fill] (3,2) circle [radius=0.07];
			\draw[fill] (3.5,0) circle [radius=0.07];
			\node[below right] at (3.5,0) {$(x_1^-,p_1^-)$};
			\draw[thick, dashed] (3.5,0) to [out=-90, in=-45] (-0.25,-0.25) to [out=135,in=-180] (0,4);
			\draw[thick] (0,4) to [out=-10, in = 80] (2,1);
			\draw[fill] (0,4) circle [radius=0.07];
			\draw[fill] (2,1) circle [radius=0.07];
			\node [above right] at (0,4) {$(x_1^+, p_1^+)$};
			\node [below] at (2,1) {$(x_f,p_f)$};
		\end{tikzpicture}
		\caption{An optimal trajectory for the HLQR problem with a single reset. If the cost is optimized over $p_0$ and $p_1^+$ independently, then the jump condition \eqref{eq:space_jump} should follow for free.}
		\label{fig:spatially_triggered}
	\end{subfigure}
	\hfill
	\begin{subfigure}[t]{0.49\textwidth}
		\centering
		\includegraphics[width=\columnwidth]{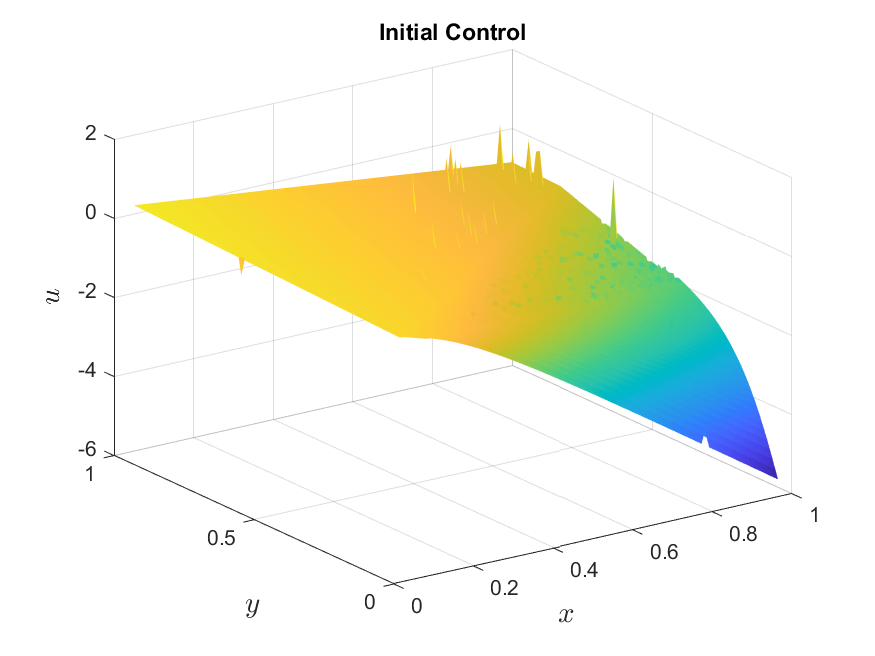}
		\caption{The feedback control law $u(t=0;x,y)$ from the weakly actuated problem in Example \ref{ex:weak}. This function is neither linear nor affine in $(x,y)$.}
		\label{fig:weak_control}
	\end{subfigure}
\end{figure}

Consider the data:
\begin{equation*}
	A = \begin{bmatrix}
		0 & 1 \\ -1 & 0
	\end{bmatrix}, \quad C = \begin{bmatrix}
		0 & 0 \\ 2 & 0
	\end{bmatrix}, \quad \lambda = \begin{bmatrix}
		0 \\ 1
	\end{bmatrix},
\end{equation*}
along with $Q = F = \mathrm{Id}_{2\times 2}$, $R = 1$, and $t_f = 1$.
\begin{example}[Weakly Actuated]\label{ex:weak}
	Let $B = [1,0]^\top$ and $x_0 = [1,0.3]^\top$. This results in weakly actuated resets as $\lambda^\top B = 0$. Optimizing over $p_0$ and $p_1^+$, we get
		\begin{equation*}
			p_0 = \begin{bmatrix}
				2.3155 \\ -1.4776
			\end{bmatrix}, \quad p_1^+ = \begin{bmatrix}
				-0.0211 \\ 0.4088
			\end{bmatrix},
		\end{equation*}
		which produces
		\begin{equation*}
			p_1^- = \begin{bmatrix}
				0.8175 \\ -2.4574
			\end{bmatrix}, \quad p_f = \begin{bmatrix}
				0.0991 \\ 0.2702
			\end{bmatrix}, \quad x_f = \begin{bmatrix}
				0.0991 \\ 0.2702
			\end{bmatrix}.
	\end{equation*}
	The boundary conditions are indeed satisfied as $p_f = Fx_f$, $H^- = 0.0365 = H^+$, and
	\begin{equation*}
		p_1^- - C^\top p_1^+ = \begin{bmatrix}
			-5.9964\cdot 10^{-5} \\ -2.4574
		\end{bmatrix} \approx -2.4574 \lambda.
	\end{equation*}
	The computed feedback control is shown in Fig \ref{fig:weak_control}. Note that the feedback control is not linear in the states. This reinforce the idea that the ansatz $p = Sx$ for spatially triggered systems is not valid.
\end{example}
\begin{example}[Strongly Actuated]\label{ex:strong}
	Let $B = [0,1]^\top$ and $x_0 = [0.5,0.2]^\top$.  The resets are strongly actuated as $\lambda^\top B \ne 0$. Optimizing over $p_0$ and $p_1^+$, we get
		\begin{equation*}
			p_0 = \begin{bmatrix}
				3.5105 \\ 1.3351
			\end{bmatrix}, \quad p_1^+ = \begin{bmatrix}
				0.2580 \\ 1.7373
			\end{bmatrix},
		\end{equation*}
		which produces
		\begin{equation*}
			p_1^- = \begin{bmatrix}
				3.4746 \\ 0.0806
			\end{bmatrix}, \quad p_f = \begin{bmatrix}
				0.7333 \\ 0.6053
			\end{bmatrix}, \quad x_f = \begin{bmatrix}
				0.7333 \\ 0.6053
			\end{bmatrix}.
	\end{equation*}
	The boundary conditions remain satisfied as $p_f = Fx_f$, $H^- = 0.2689 = H^+$, and
	\begin{equation*}
		p_1^- - C^\top p_1^+ = \begin{bmatrix}
			9.9650\cdot 10^{-6} \\ 0.0806
		\end{bmatrix} \approx 0.0806 \lambda.
	\end{equation*}
\end{example}
\subsection{Double Mass Spring System}\label{sec:dmss}
\begin{figure}
	\centering
	\begin{tikzpicture}
		\draw[thick] (0,-1) -- (0,1);
		\fill[ pattern = {Lines[distance=2mm, angle=45, line width=0.25mm]},
		pattern color=black] (-0.5,-1) rectangle (0,1);
		\draw[thick] (6,-1) -- (6,1);
		\fill[ pattern = {Lines[distance=2mm, angle=45, line width=0.25mm]},
		pattern color=black] (6,-1) rectangle (6.5,1);
		\draw[
		decoration={
			coil,
			aspect=0.3, 
			segment length=1.5mm, 
			amplitude=2mm, 
			pre length=3mm,
			post length=3mm},
		decorate
		] (0,0) -- ++(2,0) 
		node[midway,above=0.25cm,black]{$k_1$}; 
		\draw[
		decoration={
			coil,
			aspect=0.3, 
			segment length=1.5mm, 
			amplitude=2mm, 
			pre length=3mm,
			post length=3mm},
		decorate
		] (4,0) -- ++(2,0) 
		node[midway,above=0.25cm,black]{$k_2$}; 
		\node[draw,
		minimum width=0.75cm,
		minimum height=0.75cm] at (2.375,0) {$m_1$};
		\node[draw,
		minimum width=0.75cm,
		minimum height=0.75cm] at (3.625,0) {$m_2$};
		\draw[<->, thick] (2,1) -- (4,1);
		\draw[thick] (3,0.8) -- (3,1.2);
		\node[above left] at (2,1) {$x_1$};
		\node[above right] at (4,1) {$x_2$};
	\end{tikzpicture}
	\caption{A schematic of the double mass spring system presented in Section \ref{sec:dmss}.}
	\label{fig:dblspring}
\end{figure}
Consider the double mass spring system shown in Fig. \ref{fig:dblspring}. Let $x_1$ be the location of the left mass and $x_2$ be the location of the right mass measured from the center. Implementing a controlled force on the left mass yields the dynamics
	\begin{equation*}
		\begin{split}
			\ddot{x}_1 &= -k_1(x_1-d_1) + u, \\
			\ddot{x}_2 &= -k_2(x_2-d_2),
		\end{split}
	\end{equation*}
	where $d_{1,2}$ is the displacement of the spring when $x_{1,2}=0$. The masses impact when $x_1+x_2 = 0$ with
	\begin{equation*}
		\dot{x}_1^+ = -\dot{x}_2, \quad \dot{x}_2^+ = -\dot{x}_1.
	\end{equation*}
	This is an example of a 4-dimensional hybrid affine quadratic regulator with the data
	\begin{equation*}
		A = \begin{bmatrix}
			0 & 1 & 0 & 0 \\
			-k_1 & 0 & 0 & 0 \\
			0 & 0 & 0 & 1 \\
			0 & 0 & -k_2 & 0
		\end{bmatrix}, \quad b = \begin{bmatrix}
			0 \\ -k_1d_1 \\ 0 \\ -k_2d_2
		\end{bmatrix},
	\end{equation*}
	\begin{equation*}
		C = \begin{bmatrix}
			1 & 0 & 0 & 0 \\
			0 & 0 & 0 & -1 \\
			0 & 0 & 1 & 0 \\
			0 & -1 & 0 & 0
		\end{bmatrix}, \quad \lambda = \begin{bmatrix}
			1 \\ 0 \\ 1 \\ 0
		\end{bmatrix},
	\end{equation*}
	where $M = m_1+m_2$ and with $Q = 0$, $N = 0$, $F = 50\mathrm{Id}$, $R = 1$, $B = e_2$, and $y = \delta e_3$, where $e_j$ is the vector in the $j^{th}$-coordinates. The objective of this control problem is to steer the right mass by only actuating the left one.

\begin{table}
	\centering
	\begin{tabular}{r|cccccc}
		Parameter & $k_1$ & $k_2$ & $d_1$ & $d_2$ & $\delta$ & $t_f$\\ \hline
		Value & 2 & 1 & 1 & 1 & 2 & 5
	\end{tabular}
	\caption{Parameter values used in Section \ref{sec:dmss}.}
	\label{tab:param}
\end{table}

Using parameter values prescribed by Table \ref{tab:param}, a computed trajectory is shown in Fig. \ref{fig:spring_trajectory}. As this system has weakly actuated resets, the co-state jump can be unambiguously determined from Proposition \ref{prop:war}. This was computed via shooting to determine the initial co-states. The map that assigns a cost to the initial co-states is highly discontinuous as the number of resets is not fixed a priori. To see how the cost depends on the initial co-states $p_0 = [p_{x_1},p_{v_1},0,0]$, see Figs \ref{fig:spring_cost} and \ref{fig:spring_numer}.

\begin{figure}
	\centering
	\begin{subfigure}[t]{0.49\textwidth}
		\centering
		\includegraphics[width=\columnwidth]{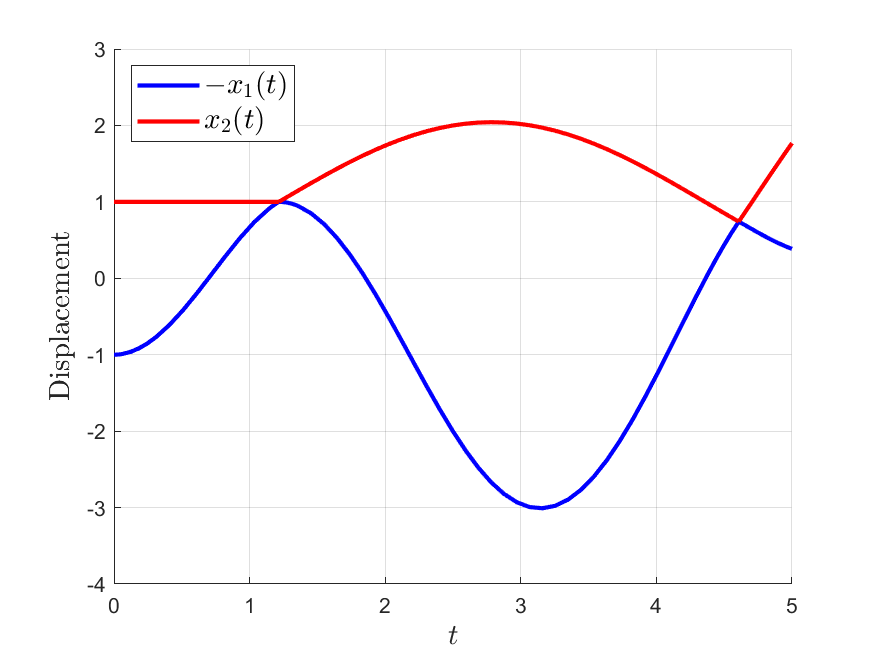}
		\caption{Computed trajectory for the double mass spring system in \S\ref{sec:dmss}.}
		\label{fig:spring_trajectory}
	\end{subfigure}
	\hfill
	\begin{subfigure}[t]{0.49\textwidth}
		\centering
		\includegraphics[width=\columnwidth]{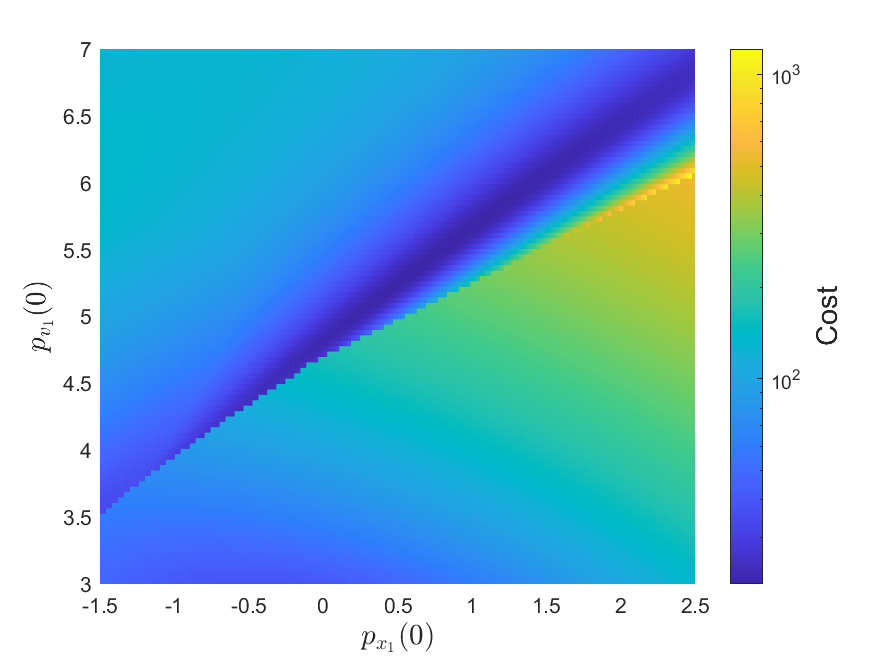}
		\caption{The cost as a function of the initial co-state value for the double mass spring system in \S\ref{sec:dmss}.}
		\label{fig:spring_cost}
	\end{subfigure}
	\\
	\begin{subfigure}[t]{0.49\textwidth}
		\includegraphics[width=\columnwidth]{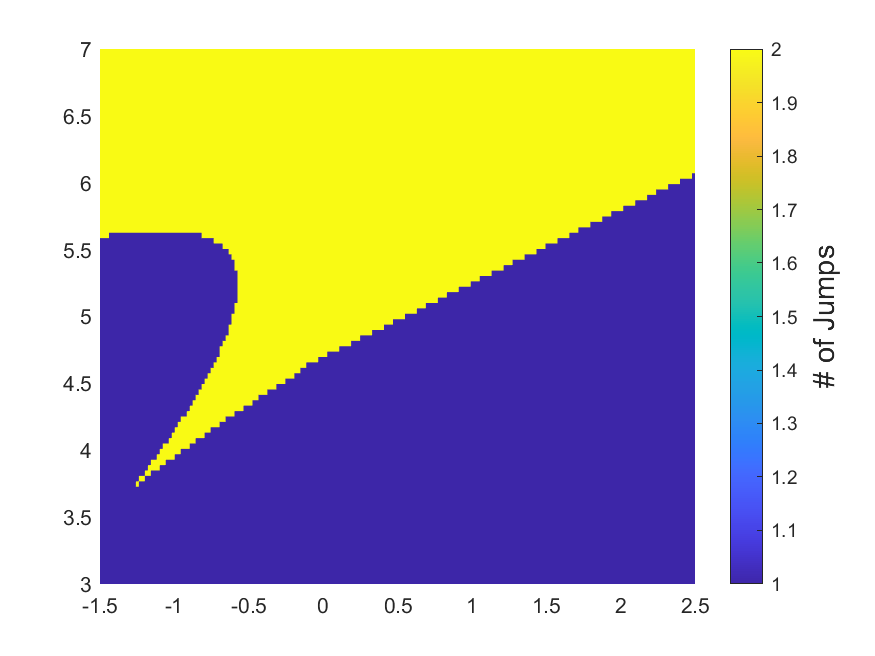}
		\caption{The number of resets that occur as a function of the initial co-state value for the double mass spring system in \S\ref{sec:dmss}.}
		\label{fig:spring_numer}
	\end{subfigure}
	\caption{Numerical results for the double mass spring system in \S\ref{sec:dmss}.}
\end{figure}

\subsection{Zeno Trajectories}\label{subsec:Zeno_ex}
Consider the following affine hybrid control system as a controlled version of Example \ref{ex:soZ}:
\begin{equation*}
	\begin{cases}
		\left. \begin{array}{l}
			\dot{x} = v \\
			\dot{v} = -1+u
		\end{array}\right\} & (x>0)\vee(x=0 \wedge v>0) \\
		\left. \begin{array}{l}
			x^+ = 0 \\
			v^+ = -0.49v^-
		\end{array}\right\} & (x=0 \wedge v<0)
	\end{cases}
\end{equation*}

	Let the cost be given by
	\begin{equation*}
		J = \int_0^{10} \, \frac{1}{2}u^2 \, dt + 10\left(x(10)-1\right)^2 + 10v(10)^2.
	\end{equation*}
	This is an affine quadratic regulator where Type II Zeno is possible and is expected for long enough time horizons as predicted by Corollary \ref{cor:Zeno_control}. Trajectories are found via dynamic programming with the discretization
	\begin{equation*}
		\begin{split}
			t &\in \texttt{linspace}(0,10,150), \\
			x &\in \texttt{linspace}(0,2,150), \\
			v &\in \texttt{linspace}(-2,2,150), \\
			u &\in \texttt{linspace}(-1,3,150).
		\end{split}
	\end{equation*}
	The computations are done by a search over all allowed controls which results in a long run-time of $\sim 18$ minutes. Indeed, a Zeno solution is found which lies outside of the predictions of the maximum principle as shown in Fig. \ref{fig:zeno_traj}. Additionally, the computed value function is constant on a region surrounding the origin which is in agreement with $J_{\text{Zeno}}(x_0,t_f)$ being constant, see Fig \ref{fig:zeno_value}.

Although this example is similar to the one presented in \cite{zeno_pmp}, the results here are computed via dynamic programming while those in \cite{zeno_pmp} are computed via shooting.

\begin{figure}
	\centering
	\begin{subfigure}[t]{0.49\textwidth}
		\centering
		\includegraphics[width=\columnwidth]{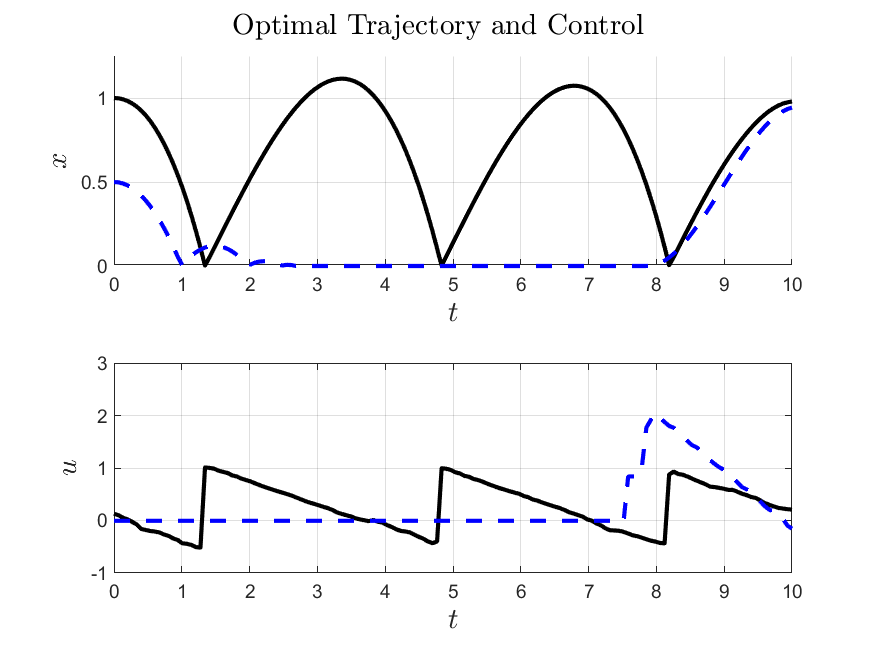}
		\caption{Two optimal trajectories for the problem in \S\ref{subsec:Zeno_ex}. The solid black curve has the initial conditions $(x(0),v(0))=(1,0)$ while the blue dashed curve has $(x(0),v(0)) = (0.5,0)$. The former admits a regular solution while the former is Zeno. }
		\label{fig:zeno_traj}
	\end{subfigure}
	\hfill
	\begin{subfigure}[t]{0.49\textwidth}
		\includegraphics[width=\columnwidth]{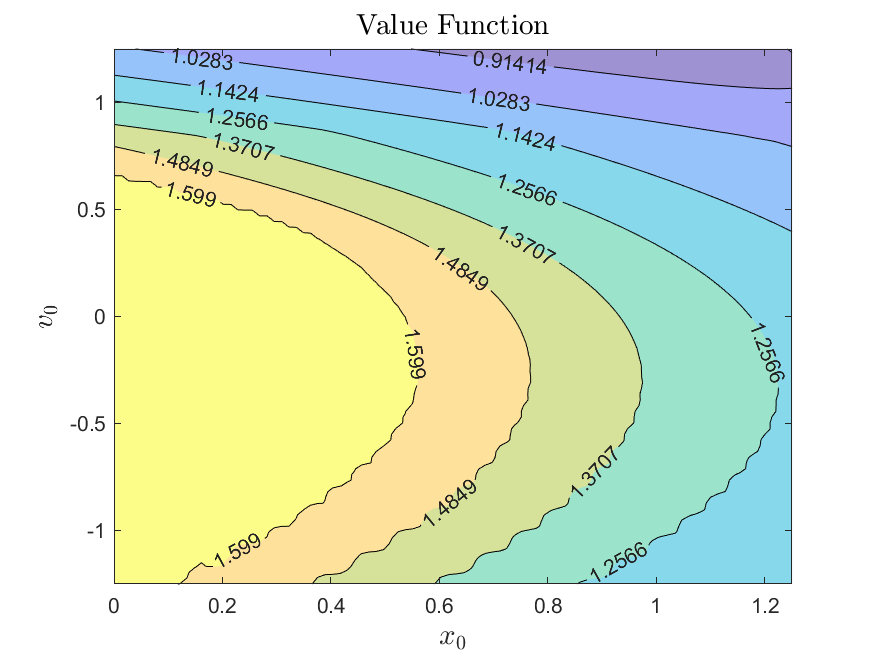}
		\caption{A contour plot for the value function corresponding to the problem in \S\ref{subsec:Zeno_ex}. The value function plateaus at the value of 1.599 which corresponds to the Zeno control strategy having constant cost.}
		\label{fig:zeno_value}
	\end{subfigure}
	\caption{Numerical results for \S\ref{subsec:Zeno_ex}. Notice that the control for the Zeno trajectory is only active over the end of the time interval in agreement with \eqref{eq:zeno_control}.}
\end{figure}

\section{Conclusions and Future Work}\label{sec:conclusion}
This work studied both the dynamics and control of affine hybrid dynamical systems. Notably, it was shown that affine hybrid dynamical systems posses exceptional arcs - beating/blocking always occur and Zeno is not uncommon. The existence of these exceptional arcs make the interpretation of the hybrid Pontryagin maximum principle difficult. Moreover, even in the regular case, the existence/uniqueness of the co-state jump condition is not straightforward. Although there exists a plethora of open issues, three principal topics are:

\begin{enumerate}
	\item Although optimal trajectories for linear/affine hybrid systems can be found by solving \eqref{eq:variational_corner}, reasonable numerical schemes need to be developed. In particular, techniques to solve for stabilizing solutions to \eqref{eq:triggered_riccati} for temporally triggered systems and \eqref{eq:space_jump} for spatially triggered systems. Moreover, although optimal trajectories with beating can be found by solving \eqref{eq:kbeating_boundary}, a reasonable numerical scheme needs to be developed.
	\item Although blocking and Zeno can be successfully excluded from linear systems, this is not the case for affine (and especially non-linear) systems. This leads to questions in both dynamics and controls.
	\begin{enumerate}
		\item From a dynamics viewpoint: How do the sets $\Sigma_\infty$ and $\Sigma_\infty^f$ behave and what is their structural stability? Here, $\Sigma_\infty^f$ is the invariant blocking set defined to be the nonlinear analog of $\Sigma_\infty^A$ in \eqref{eq:invariant_beating}.
		\item From a control viewpoint: The hybrid maximum principle breaks down when Zeno occurs which does occur in examples. A more general theory should be developed to handle this situation. One potential approach is to develop Zeno trajectories as weak solutions to the associated Hamilton-Jacobi-Bellman problem by examining the limit of $J_\varepsilon$ as $\varepsilon\to 0$.
	\end{enumerate}
	\item Strange behavior is seen when arcs terminate on the guard. The trajectory found in the strongly actuated case of Example \ref{ex:strong} with initial condition $(0.75,0.5)$ terminates on the guard with final state and co-state
		\begin{equation*}
			x_f = \begin{bmatrix}
				0.2872 \\ 8.9361\cdot 10^{-9}
			\end{bmatrix}, \quad p_f = \begin{bmatrix}
				0.2874 \\ -0.5438
			\end{bmatrix}.
		\end{equation*}
		The condition $p_f = Fx_f$ does not hold here.
\end{enumerate}
\section*{Acknowledgement}
The author greatly benefited from feedback by M. Ruth (Oden Institute). Comments from A. Bloch and M. Ghaffari (University of Michigan) additionally imporoved this work. This work was supported by AFOSR Award No. MURI FA9550-23-1-0400.

\bibliography{sn-bibliography}

\end{document}